\documentclass[12pt,reqno]{article}

\usepackage{amsthm,amsfonts,amssymb,amsmath,epsf, verbatim}
\newtheorem{theorem}{Theorem}
\newtheorem{lemma}[theorem]{Lemma}
\newtheorem{corollary}[theorem]{Corollary}
\newtheorem{proposition}[theorem]{Proposition}

\newtheorem{conjecture}[theorem]{Conjecture}

\begin{document}
\title{On the number of total prime factors of an odd perfect number}

\author{Joshua Zelinsky}
\date{}
\maketitle
\vspace{-1 cm}
\begin{center}
Iowa State University\\
Email:zelinsky@gmail.com
\end{center}
\begin{abstract} Let $N$ be an odd perfect number. Let $\omega(N)$ be the number of distinct prime factors of $N$ and let $\Omega(N)$ be the total number of prime factors of $N$. We prove that if $(3,N)=1$, then  $ \frac{302}{113}\omega - \frac{286}{113} \leq  \Omega. $
 If $3|N$, then $\frac{66}{25}\omega-5\leq\Omega.$ This is an improvement on similar prior results by the author which was an improvement of a result of Ochem and Rao. We also establish new lower bounds on $\omega(N)$ in terms of the smallest prime factor of $N$ and establish new lower bounds on $N$ in terms of its smallest prime factor.
\end{abstract}

Recall that a positive integer $N$ is said to be perfect if the sum of $N$'s proper divisors add up to the number or equivalently that $\sigma(N)=2N$.

It is currently unknown whether there are any odd perfect numbers. Let $N$ be an odd perfect number. Ochem and Rao\cite{OchemRao1}  have proved that $N$ must satisfy \begin{equation}\Omega(N) \geq \frac{18\omega(N) -31}{7}\label{OR1}\end{equation} and \begin{equation}\Omega(N) \geq 2\omega(n) +51.\label{OR2} \end{equation} Note that Ochem and Rao's second inequality is stronger than the first as long as $\omega(N) \leq 81$. Nielsen \cite{Nielsen} has shown that $\omega(n) \geq  10 $.
In a previous paper \cite{Zelinsky1}, the author improved on Ochem and Rao's result, showing that:

\begin{theorem} If $N$ is an odd perfect number, with $3\not| N$, then \begin{equation}\Omega(N) \geq \frac{8}{3}\omega(N)-\frac{7}{3}.\label{firstineq}\end{equation}   If $N$ is an odd perfect number, with $3 |N$, then \begin{equation}\Omega(N) \geq \frac{21}{8}\omega(N)-\frac{39}{8}.\label{secondineq}\end{equation}   \label{JZ Theorem from Previous}
\end{theorem}
In this paper we prove
\begin{theorem}\label{OR stage II}
If $3 \not |N$, then \begin{equation}\label{New OR  bound for (3,N)=1}\Omega(N) \geq \frac{302}{113}\omega(N) - \frac{286}{113} \end{equation}
 If $3|N$, then  \begin{equation}\label{New OR bound for 3|N} \Omega(N) \geq \frac{66}{25}\omega(N)-5\end{equation}
\end{theorem}
Note that while Inequality \ref{New OR bound for 3|N} is always better than  Inequality \ref{secondineq}, Inequality \ref{New OR  bound for (3,N)=1} is only better than Inequality \ref{firstineq} when $\omega \geq 34$.

Note that the worst case of the above is when $3|N$, and so we have
\begin{corollary}If $N$ is an odd perfect number then $$\Omega(N) \geq \frac{66}{25}\omega(N)-5.$$
\end{corollary}
Note that Kevin Hare \cite{Hare} has shown that in general any odd perfect number must satisfy $\Omega \geq 75$, while \cite{OchemRao2012} has improved this to $\Omega \geq 101$.
This paper will contain seven sections. The first section contains various results we will need to prove Theorem \ref{OR stage II}. The second section contains the proof of Theorem \ref{OR stage II} when $3|N$. The third section contains the proof when $(3,N)=1$.  The fourth section improves on the known lower bound of $\Omega$ in terms of the smallest prime factor of $N$. This is essentially a small improvement of existing results although new questions are raised based on some aspects of the methods used. The fifth section combines the ideas of the previous sections to improve lower bounds for $N$ in terms of its smallest prime factor. The sixth section discusses a new way of measuring how strong a statement is about odd perfect numbers and evaluates the Ochem and Rao type bounds in that context. The seventh section discusses various related open problems that are naturally connected to improving these results.
We will use the following notation: $N$ will be an odd perfect number.  We will write $\Omega$ for $\Omega(N)$ and write $\omega$ for $\omega(N)$.
We recall Euler's classical theorem on odd perfect numbers.   Euler proved that $N$ must have the form $N=q^e m^2$ where $q$ is a prime such that $q \equiv e \equiv 1$ (mod 4) and  $(q,m)=1$. Traditionally $q$ is called the special prime.\footnote{Some authors call $q$ the ``Euler prime.'' A better name in fact would be the Cartesian prime since prior to Euler's result Descartes proved that an odd perfect number needed to have exactly one prime factor raised to an odd power. In any event, special prime avoids any issues of priority.}  Note that from Euler's result one immediately has $\Omega \geq 2\omega - 1$. Essentially all improvements on Ochem-Rao type inequalities can be thought of as improving on the bound one has from Euler's theorem.  For the remainder of this paper we will assume that $N$ is an odd perfect with $q$, $e$ and $m$ given  as above.
The basic method of this paper is the same as that used in Ochem and Rao's result.  The essential observation is that if $N$ is an odd perfect number with a prime $p$ raised to just the second power, then for each such prime $p$ you have $p^2+p+1|n$. It follows from quadratic reciprocity that if $q$ is a prime and $q|p^2+p+1$, then $q$ is either equal to 3 or is 1 mod 3. Now assume that $N$ has many such primes $p$. If $p^2+p+1$ is often divisible by 3, then $N$ will be divisible by a large power of $3$. If q is not 3 in some instance, than with the exception of the one prime in an OPN which may be raised to the first power, one either has  $q^2||N$ , which gives a new 3 in the factorization since $q \equiv  1$ (mod 3), and so $3|q^2+q+1$, or one at least has $q^4|N$. So the key then is that if one has many numbers of the form $p^2+p+1$ showing that there isn't much room for a lot of primes repeated exactly twice.  And if one has most primes raised to a higher power then one gets repeated prime factors from those primes.  This idea is made rigorous through a system of linear inequalities and optimizing that set of inequalities then gives the type of result. While this paper is substantially longer than the previous paper by this author or Ochem and Rao's paper, the basic method remains the same.
The improvements in the case of $3|N$ are essentially straightforward and represent a small improvement of the technique using some new minor number theoretic results to get additional inequalities in the system used. The case of $(3,N)=1$ involves three major new ingredients.
Our first new ingredient is the notion of a triple threat. Define a {\emph{triple threat}}, to be a quadruplet of four odd primes $x,a,b,c$ such that $$\sigma(x^2)=x^2+x+1=\sigma(a^2)(\sigma(b^2)\sigma(c^2)=(a^2+a+1)(b^2+b+1)(c^2+c+1)$$ and where $a^2+a+1$, $b^2+b+1$, and $c^2+c+1$ are all prime. The primary obstruction to improving the $8/3$ bound in Theorem  \ref{JZ Theorem from Previous} arose from the fact that we could not rule out the existence of odd perfect numbers with many primes acting triple threats. Consider a triple threat $(x,a,b,c)$ where $x \equiv a  \equiv 1$ (mod 5).  We show that no such triple threats exist with this property. This is not by itself sufficient to improve the bound.
Our second new ingredient is that if $p^4||N$, then every prime divisor of $\sigma(p^4)$ is either equal to 5 or is $1$ (mod 5).
Our third new ingredient is the observation that we have the following miraculous factorization:\footnote{This observation seems to have been first noted explicitly in the literature in \cite{Fletcher Nielsen and Ochem}.} If $f(x)=x^2+x+1$, and $g(x)=x^4+x^3+x^2+x+1$, then $$f(g(x))=(x^2-x+1)(x^6+3x^5+5x^4+6x^3+7x^2+6x+3).$$ This means that if we have a prime $m$ such that $m^4||N$, and $\sigma(m^4) = p$ is itself prime then we can substantially restrict what $\sigma(p^2)$ looks like. We can combine the second and third ingredients to to guarantee that there are either many primes $p$ where $p^2||N$ and $p \equiv 1$ (mod 5), $N$ is divisible by a large power of $5$, $N$ has many prime factors raised to at least the sixth power, or $N$ has many prime factors of a very special form arising from the third ingredient factorization.
Note that all three of these ingredients are necessary for our improvement. Any two of them will not by themselves give rise to an improvement beyond the $8/3$ bound.
\section{Foundations}
This section contains various lemmata we will need for the main results.  We will assume some basic familiarity with the literature on perfect numbers but will recall some basic facts here. For some early history on this matter see \cite{Gimbel}. \\
In general, while the Greeks originally defined perfect numbers in terms of the sum of the proper divisors,  it is more natural to define a number as perfect as $n$ satisfying $\sigma(n)=2n$ where $\sigma(n)$ is the sum of all the positive divisors of $n$. Much of the study of perfect numbers relies on the nice fact that $\sigma(n)$ is a multiplicative function.  Recall that a number is said to be {\emph{abundant}} if $\sigma(n)>2n$.  A number is said to be {\emph{deficient}} if $\sigma(2n)<n$. We will set $h(n)= \frac{\sigma(n)}{n}$. Note that $h(n)$ is called both the {\emph{abundancy}} of $n$ or the {\emph{index}} of $n$ in the literature.  We have that \begin{equation}\label{abundancy} h(n) = \frac{\sum_{d|n} d}{n} = \sum_{d|n} \frac{d}{n} = \sum_{d|n} \frac{1}{d}. \end{equation}  From \ref{abundancy} we have that if $a>1$, then we have $h(an)>h(n)$. In particular, any multiple of an abundant number is itself abundant, and thus no perfect number can be divisible by an abundant number.
Almost all work on odd perfect numbers relies on the fundamental observation that we can bootstrap from knowing that a specific prime power divides $N$ to get that other prime powers divide $N$. For example, if $3^2||N$, then $\sigma(3^2)=13$ must divide $N$. In general, if $p^k||N$ then we must have
$\sigma(p^k)=1+p+p^2 \cdots p^k|2N$. For any $k$ we have $$1+x+x^2... + x^k = \frac{x^{k+1}-1}{x-1}$$ and we may factor $\frac{x^{k+1}-1}{x-1}$ into cyclotomic polynomials. Thus, a major part of understanding odd perfect numbers comes from understanding the integer values of cyclotomic polynomials.

\begin{lemma}
If $a$ and $b$ are distinct odd primes and $p$ is a prime such that $p|(a^2+a+1)$ and $p|(b^2+b+1)$. If  $a \equiv b \equiv 2$ (mod 3), then $p \leq \frac{a+b+1}{5}$.  If $a \equiv b \equiv 1$ (mod 3) then $p \leq \frac{a+b+1}{3}$.
\label{smallsharedprimes}
\end{lemma}
This is Lemma 1 from \cite{Zelinsky1}
We will also need the following result, which is Lemma 3 in Ochem and Rao:
\begin{lemma} Let $p$, $q$ and $r$ be positive integers. If $p^2+p+1=r$ and $q^2+q+1=3r$ then $p$ is not an odd prime.
 
\end{lemma}

\begin{lemma} If $x$ is a positive integer then the only possible common prime divisor of $x^2-x+1$ and $x^6+3x^5+5x^4+6x^3+7x^2+6x+3$ is 31.
\label{Relative primeness of two parts of T to S1}
\end{lemma}
\begin{proof} Set $A=x^2-x+1$ and $B=x^6+3x^5+5x^4+6x^3+7x^2+6x+3$. If $p|(A,B)$ then $p$ divides any linear combination of them. In particular, $$p|(5x^5+21x^4+44x^3+58x^2+55x+34)A- (5x+1)B=31.$$ So $p=31$.
\end{proof}
\begin{lemma} The only non-negative integer solutions to the equation  $$x^6+3x^5+5x^4+6x^3+7x^2+6x+3=a^2+a+1 $$ $(x,a) = (0,1)$ and $(x,a)=(1,5)$.
\label{Big piece of T1 to S is an S1 cannot happen}
\end{lemma}
\begin{proof} We can verify that by direct computation that these two solutions are the only solutions where $0 \leq 26 \leq x$, so we may assume that $x \geq 27$. Some algebra shows that we may write $a$ in terms of $x$ as \begin{equation}a=x^3+\frac{3}{2}x^2+\frac{11}{8}x+\frac{7}{16}+t
\label{a in terms of x}
\end{equation}
where $$t=\frac{588x^2+876x+351}{256x^3+384x^2+352x+240+128\sqrt{4x^6+12x^5+20x^4+24x^3+28x^2+24x+9}}.$$
For any integer $x$, $x^3+\frac{3}{2}x^2+\frac{11}{8}x$ is a rational number whose denominator divides $8$, and the next term in \ref{a in terms of x} is $7/16$, so the only way that $a$ can be an integer is if $t$ is a fraction whose denominator is $16$. However, we have that  $$0< t < \frac{876x^2}{256x^3+384x^2}=\frac{219}{128x+96} < \frac{1}{16}.  $$ Here the last inequality on the right is due to $x \geq 27$.  Since $t$ is strictly between $0$ and $1/16$ it cannot be a fraction with denominator 16, and so there are no more solutions.
\end{proof}
We then have also
\begin{lemma} The only positive integer solution to $$x^6+3x^5+5x^4+6x^3+7x^2+6x+3=a^2-a+1 $$ is $(x,a)=(1,6)$.
\label{Big piece of T1 to S is from a small T1 to S cannot happen}
\end{lemma}
\begin{proof} This follows from Lemma \ref{Big piece of T1 to S is an S1 cannot happen} and noting that this equation is identical to the equation from that lemma but with $a-1$ substituted for $a$.
\end{proof}
\begin{lemma}\label{T1 to S2 contributes an 11} Let $x$ be a prime with $\sigma(x^4)= x^4+x^3+x^2+x+1$ prime, and suppose that $a=\sigma(x^4)$ is prime.  Then $\sigma(a^2)=a^2+a+1$ has at least two distinct prime factors.
Furthermore, if $\sigma(a^2)=bc$ for two distinct primes $b$ and $c$ then either $11|\sigma(b^4)$ or $11|\sigma(c^4)$.
\end{lemma}
\begin{proof} Assume that $x$ is prime, and assume further that $a=\sigma(x^4)$ is prime. Then $a=x^4+x^3+x^2+x+1$ and we have $\sigma(a^2)=a^2+a+1$. A straightforward calculation gives us \begin{equation} \label{T1 S factorization first use}
    \begin{split}
    \sigma(a^2) & =x^8+2x^7+3x^6+4x^5+6x^4+5x^3+4x^2+3x+3\\
    &= (x^2-x+1)(x^6+3x^5+5x^4+6x^3+7x^2+6x+3).
\end{split}
\end{equation}
By lemma \ref{Relative primeness of two parts of T to S1} we have that $\sigma(a^2)$ must have at least two distinct prime factors, unless both $(x^2-x+1)$ and $x^6+3x^5+5x^4+6x^3+7x^2+6x+3$ are powers of 31. But they cannot both be a power of 31. To see why note, that this would make $(x^2-x+1)^3$ also a power of 31, and we would have a contradiction due to the fact that as long as $x>2$ we have $$(x^2-x+1)^3 < x^6+3x^5+5x^4+6x^3+7x^2+6x+3 < 31(x^2-x+1)^3.$$
To prove the last part of this Lemma, we now note that if $\sigma(a^2)=bc$ for two primes $b$ and $c$ we must have that one of the primes is $x^2-x+1$ and the other prime is $x^6+3x^5+5x^4+6x^3+7x^2+6x+3$. Without loss of generality, let us assume that $b=x^2-x+1$ and that $c=x^6+3x^5+5x^4+6x^3+7x^2+6x+3$. It is easy to check that for any modulus we have that either $b^4+b^3+b^2+b+1$ or $c^4+c^3+c^2+c+1$ is $0$ (mod 11).
\end{proof}
\begin{lemma}\label{Triple threats cannot come from T1 to S3} There are no odd primes primes $x$, $a$, $b$, $c$, $d$ satisfying the conditions:
\begin{enumerate}
    \item $a=\sigma(x^4)$
    \item $\sigma(a^2)=\sigma(b^2)\sigma(c^2)\sigma(d^2)$
    \item $\sigma(b^2)$, $\sigma(c^2)$, and $\sigma(d^2)$ are all prime.
\end{enumerate}
\end{lemma}
\begin{proof}Assume that we have such a solution. Then by same logic as in the proof of Lemma \ref{T1 to S2 contributes an 11}
\begin{equation} \label{T1 S factorization second use}
    \begin{split}
    \sigma(a^2) & =x^8+2x^7+3x^6+4x^5+6x^4+5x^3+4x^2+3x+3\\
    &= (x^2-x+1)(x^6+3x^5+5x^4+6x^3+7x^2+6x+3).
\end{split}
\end{equation}
We have that $\sigma(b^2)=b^2+b+1$, $\sigma(c^2)=c^2+c+1$, and $\sigma(d^2)=d^2+d+1$. Since all three of these quantities are prime we must have one of them equal to either $x^2-x+1$ or $x^6+3x^5+5x^4+6x^3+7x^2+6x+3$. Without loss of generality, we will assume that this is $b^2+b+1$. We have two cases: either $b^2+b+1=x^6+3x^5+5x^4+6x^3+7x^2+6x+3$ or $b^2+b+1=x^2-x+1$. The first case is ruled out by Lemma \ref{Big piece of T1 to S is an S1 cannot happen} so we must have $b^2+b+1=x^2-x+1$. Note that $x^2-x+1=(x-1)^2 + (x-1) +1$ Thus,  $b^2+b+1= (x-1)^2+(x-1)+1$ and so $b=x-1$ (since $t^2+t+1$ is a strictly increasing function for $t> 1/2$. But $b=x-1$ is impossible since $b$ and $x$ are both odd primes.
\end{proof}
\begin{lemma}\label{No triple threat arising from T} We cannot have integers $a,b,c,d$ with $a \equiv b \equiv c \equiv d \equiv 1$ (mod 5) and also satisfying $a^2+a+1=(b^2+b+1)(c^2+c+1)(d^2+d+1)$.
\end{lemma}
\begin{proof} This just follows from observing that the left side of the equation is $3$ (mod 5) and the right side is $2$ (mod 5).
\end{proof}

Let $(x,a,b,c)$ be a quadruplet of odd primes all greater than 3. We say that they form a \emph{triple threat} if they satisfy two conditions:
\begin{enumerate}
    \item We have the relationship $$ x^2+x+1 = (a^2+a+1)(b^2+b+1)(c^2+c+1).$$
    \item $a^2+a+1$, $b^2+b+1$, and $c^2+c+1$ are all prime.
\end{enumerate}
If we could show that there are no triple threats in general, then we could substantially improve our bounds in this paper for both the $3|N$ case and the $3 \not|N$ case. However, we are presently unable to do that, and so we must satisfy ourselves with instead proving substantial enough restrictions on what triple threats can look like.  We will in particular prove that we cannot have $x \equiv 1$ (mod 5) while also having one of $a$, $b$ or $c$ also $1$ (mod 5).  
Assume that $(x,a,b,c)$ is a triple threat, and that $x \equiv a \equiv 1$ (mod 5). Then without loss of generality one must have \begin{equation}\label{triple threat options mod 5} (c,d)  \in \{(1,2),(2,3),(4,4) (\mathrm{mod\, } 5)  \}. \end{equation}
We will rule out each of these three options separately. We do by proving various statements about the equation $x^2+x+1=(a^2+a+1)(b^2+b+1)p$ with $x$, $a$, $b$, $a^2+a+1$, $b^2+b+1$ and $p$ all prime. This method of approach has  advantages. First, while triple threats do not seem to exist, solutions of this equation do exist. Thus we will at least be proving statements about actual, mathematical objects.  Second, this equation appears to be of natural interest for extending results beyond this paper, as will be discussed later.

\begin{lemma} Assume that $x^2+x+1 = (a^2+a+1)(b^2+b+1)p$  with   $x$, $a^2+a+1$, $b^2+b+1$,$a$ and $b$ and $p$ all primes greater than 3. Assume also that $a \leq b$. Then we have one of four possibilities:
\begin{enumerate}
    \item  We have $a^2+a+1|x-a$, $b^2+b+1|x-b$, $x+a+1|p(a+b+1) + \frac{x-b}{b^2+b+1}$, and  $x+b+1|p(a+b+1) + \frac{x-a}{a^2+a+1}$.
    \item  We have $a^2+a+1|x-a$,  $b^2+b+1|x+b+1$, and $x+a+1| p(b-a) - \frac{x+b+1}{b^2+b+1}$.
    \item  We have $a^2+a+1|x+a+1$ and $b^2+b+1|x-b $, and $x+b+1|p(b-a) + \frac{x+a+1}{a^2+a+1}$.
    \item  We have $a^2+a+1|x+a+1$ and $b^2+b+1|x+b+1$ and  $x-a|p(a+b+1)-\frac{x+b+1}{b^2+b+1}>0$ and $x-b|p(a+b+1)-\frac{x+a+1}{a^2+a+1}$
\end{enumerate}
In all four cases, the quantities on the right hand sides of the divisibility relations are positive.
\end{lemma}
\begin{proof} Assume that  $x^2+x+1 = (a^2+a+1)(b^2+b+1)p$  with   $x$, $a^2+a+1$, $b^2+b+1$, $a$,  $b$ and $p$ all primes greater than 3. Then
$$(x-a)(x+a+1)=x^2+x+1 - (a^2+a+1)=(a^2+a+1)(p(b^2+b+1)-1)$$ and  $$(x-b)(x+b+1)=x^2+x+1- (b^2+b+1)=(b^2+b+1)(p(a^2+a+1)-1).$$ Since $a^2+a+1$ and $b^2+b+1$ are prime  we have four situations:
\label{four case lemma}
\begin{enumerate}
\item $a^2+a+1|x-a$ and $b^2+b+1|x-b$.
\item $a^2+a+1|x-a$ and $b^2+b+1|x+b+1$.
\item $a^2+a+1|x+a+1$ and $b^2+b+1|x-b$.
\item $a^2+a+1|x+a+1$ and $b^2+b+1|x+b+1$.
\end{enumerate}
Note that it is easy to check that we have that $(x-a, x+a+1)=(x-b,x+b+1)=1$. Consider each of these four situations:
Case I: We have $a^2+a+1|x-a$ and $b^2+b+1|x-b$. Set  $k_a=\frac{x-a}{a^2+a+1}$ and $k_b = \frac{x-b}{b^2+b+1}.$  Note that $x+a+1|p(b^2+b+1)-1$ and $x+b+1|p(a^2+a+1)-1$. Note that $$x+a+1 = k_a(a^2+a+1)+2a+1 = k_b(b^2+b+1)+a+b+1$$ and
$$x+b+1 = k_a(a^2+a+1)+a+b+1 = k_b(b^2+b+1)+2b+1.$$ We have then
$$x+a+1|p(x+a+1)-k_b(p(b^2+b+1)-1) = p(k_b(b_2+b+1)+a+b+1)-k_b(p(b^2+b+1)-1).$$ We then note that $$p(k_b(b_2+b+1)+a+b+1)-k_b(p(b^2+b+1)-1)=  p(a+b+1) + k_b  = p(a+b+1) + \frac{x-b}{b^2+b+1}.  $$
The other relation then follows by symmetry. \\
Case II: We have $a^2+a+1|x-a$ and $b^2+b+1|x+b+1$. Set $$k_a=\frac{x-a}{a^2+a+1}$$ and $$k_b=\frac{x+b+1}{b^2+b+1}.$$
We have then that $x+a+1|p(b^2+b+1)-1$.
We have $$x+a+1 = k_a(a^2+a+1)+2a+1=k_b(b^2+b+1)+a-b.$$
$$x+a+1 |k_b(p(b^2+b+1) -1) - p(k_b(b^2+b+1+1)-b+a) = -(p(b-a) -k_b).$$
We have then $x+a+1|p(b-a) - \frac{x+b+1}{b^2+b+1}$. We need to show that $p(b-a) - \frac{x+b+1}{b^2+b+1} > 0$. Assume that $p(b-a)-k_b \leq 0.$ Note in this case we have $b \neq a$, and so $p(b-a) \leq  \frac{x+b+1}{b^2+b+1}$. Since $a \equiv b \equiv 2$ (mod 3), we have that $b-a \geq 6$. Thus, $6p \leq \frac{x+b+1}{b^2+b+1}$, and we get that $2pb^2 <x$. Then
$$2pa^2 < 2pb^2 < x.$$ We then have
$$4p^2a^2b^2 < x^2 < x^2+x +1 = pa^2b^2,$$ which is a contradiction.
Thus, $p(b-a)-k_b$ is positive.
Case III: We have that $$a^2+a+1|x+a+1$$ and $$b^2+b+1|x-b.$$ We set
$k_a=\frac{x+a+1}{a^2+a+1}$ and $k_b = \frac{x-b}{b^2+b+1}.$ We have then from logic identical to that in Case II that $x+b+1|p(a-b)-k_a$. The right hand side is negative, so  $x+b+1|p(b-a) + k_a$ is positive.\\

Case IV: We have $a^2+a+1|x+a+1$ and $b^2+b+1|x+b+1$.\\
We have then $x-a|p(x-a) - k_b(p(b^2+b+1)-1$ where $k_b=\frac{x+b+1}{b^2+b+1}$. From $x-a=k_b(b^2+b+1)-a-b-1$  we get that $x-a|p(a+b+1)-k_b$.  The other
case follows from similar reasoning.  Positivity follows from an argument similar to the case in Case II.

\end{proof}

\begin{lemma} Assume that $x^2+x+1 = (a^2+a+1)(b^2+b+1)p$, with $x$, $a^2+a+1$, $b^2+b+1$,$a$ and $b$ and $p$ all primes greater than 3  and $a \leq b$.  .  Then $2b < p$.
\label{2blessthanp}
\end{lemma}
\begin{proof}
Note that we must have that $x \equiv a \equiv b \equiv 2$ (mod 3)
We split up into each of the four cases from Lemma \ref{four case lemma}\\

Case I:
We have that $x+a+1|p(a+b+1)+k_b$  where $k_b=\frac{x-b}{b^2+b+1}$. Thus
$x+a+1 \leq p(a+b+1) + k_b$.
This gives us that
$$k_b(b^2+b+1) +a + b +1 \leq p(a+b+1)+k_b.$$
This is the same as
$$\frac{k_b b^2}{a+b+1} + \frac{k_b b}{a+b+1} +1 \leq p.$$ Note that $a+b+1 \leq 2b$ and that congruence arguments give us that $k_b \geq 6$. We have in this case that $a \neq b$,  and so $a+b+1 < 2b +1 <3b$. We get that $$2b+2 <  \frac{k_b b^2}{a+b+1} + \frac{k_b b}{a+b+1} +1 \leq p.$$  \\
Case II:

We have  $x+a+1 |p(b-a) -k_b$ with $k_b=\frac{x+b+1}{b^2+b+1}$. Note that $p(b-a)-k_b$ is positive.
We have that $p(b-a)-k_b \equiv 1 $ (mod 6), and $x+a+1 \equiv 5$ (mod 6). Therefore  $h(x+a+1)=p(b-a)-k_b$ for some  $h \geq 5$.
Thus we have
$$5(x + a + 1) \leq p(b-a) - k_b.$$
Since $k_b = \frac{x+b+1}{b^2+b+1} = \frac{x}{b^2+b+1} + \frac{b+1}{b^2+b+1}$ and $b>3$, we have that $k_b < x/12$. Thus
$$5x \leq p(b-a) - x/7.$$
Then $\frac{36}{5}x \leq p(b-a)$. Since $b^2 + b +1 \leq x$, the desired inequality follows.
Case III: Case III is very similar to case II. We have that  $x+b+1|p(b-a) + k_a$ where $k_a=\frac{x+a+1}{a^2+a+1},$
We have that $x+b+1 \leq p(b-a) +k_a$. We have $k_a \leq x/13$, and so we have that
$$(12/13)x \leq p(b-a) - b -1.$$ From $b^2+b+1|x-b,$ we get that $b^2+b+1 \leq \frac{x}{6}$ which gives us the desired inequality.\\
Case IV:
We have then $x-a|p(x-a) - k_b(p(b^2+b+1)-1$ where $k_b=\frac{x+b+1}{b^2+b+1}$. From $x-a=k_b(b^2+b+1)-a-b-1$  we get that $x-a|p(a+b+1)-k_b$. Thus
$k_b(b^2+b+1)-a-b-1 \leq p(a+b+1)-k_b$  and so we get that from $k_b \geq 5$ and $k_bb^2 <p(a+b+1)$ that $2b<p$.
\end{proof}

We get as an immediate corollary:
\begin{corollary}  Assume that $x^2+x+1 = (a^2+a+1)(b^2+b+1)p$. Assume that $a^2+a+1$, $b^2+b+1$, and $p>3$ are all prime. Then $p>x^{2/5}$.
\label{Maximum obstruction size}
\end{corollary}
\begin{proof}
Assume that $x^2+x+1 = (a^2+a+1)(b^2+b+1)p$, with $a^2+a+1$, $b^2+b+1$, and $p$ prime. As usual, assume that $a \leq b$. Then we have that $x^2+x+1 = p(a^2+a+1)(b^2+b+1)< p(\frac{p^2}{4}+\frac{p}{2}+1)(\frac{p^2}{4}+\frac{p}{2}+1) \leq \frac{3}{4}p^5$ from which the desired inequality follows.
\end{proof}
\begin{lemma} If $x^2+x+1 = (a^2+a+1)(b^2+b+1)p$, with $a^2+a+1$, $b^2+b+1$, $a$, $x$ and $p$ are all primes prime greater than 3, and $b \geq a \geq 5$, then $a^2+a+1 < \frac{49}{4}p$.
\label{49over4p}
\end{lemma}
\begin{proof}
This proof requires breaking down the cases above in further detail. In Case I, $x+b+1 \leq p(a+b+1)+k_a$.  Note that $k_a \leq \frac{x}{7}$ which yields $$\frac{6}{7}x \leq p(a+b+1)-b-1.$$ So $$x \leq \frac{7}{6}(p(a+b+1)-b-1) \leq \frac{7}{6}p(3b)-1 = \frac{7}{2}pb-1.$$ We have then that
$$p(a^2+a+1)b^2 < p(a^2+a+1)(b^2+b+1) = x^2+x+1 \leq (\frac{7}{2}(pb-1))^2 + \frac{7}{2}(pb)-1 +1 < \frac{49}{4}p^2b^2$$ which implies that $a^2+a+1 < \frac{49}{4}p.$\\
The other cases are similar.
\end{proof}
\begin{lemma}
If $x^2+x+1 = (a^2+a+1)(b^2+b+1)p$, with $a^2+a+1$, $b^2+b+1$, and $p>3$ prime and $b \geq a \geq 3$, then $a^2+a+1 < (25/2)^{1/3}x^{2/3} \leq 3x^{2/3}$.
\end{lemma}
\begin{proof}
We have $$x^2+x+1 = p(a^2+a+1)(b^2+b+1) \geq \frac{4}{49}(a^2+a+1)^3$$ and so
$$(50/4)x^2 \geq (a^2+a+1)^3$$ which leads to the desired inequalities.
\end{proof}
Note that this means that we can take $k_a \geq \frac{x^{1/3}}{3}$ in all cases of Lemma \ref{2blessthanp} and can get tighter versions of that result. 

\begin{lemma} Assume that $x$ is a prime. Then $x^2+x+1$ is not a perfect cube.
\label{x2+x+1 not a cube of a prime}
\end{lemma}
\begin{proof} Assume that $x$ is prime and $x^2+x+1=k^3$. Then we have that $$x(x+1) = k^3-1 = (k-1)(k^2+k+1).$$
Since $x$ is prime, either $x | k-1$ or $x | k^2+k+1.$
If $x | k-1$, then $x \leq k-1$ and $k^3 \geq (x+1)^3 > x^2+x+1 = k^3$, which is
impossible. Now consider the possibility that $x | k^2+k+1.$ Since $x$ is prime, then either $x = 3$, which does not give a solution , or $x \equiv 1$ (mod 3).
Thus, $x^2+x+1 \equiv 3$ (mod 9), but no perfect cube is congruent to $3$ (mod 9) and so we again reach a contradiction.
\end{proof}

\begin{lemma}  Assume that $x^2+x+1 = (a^2+a+1)(b^2+b+1)p$, with $a^2+a+1$, $b^2+b+1$, $a$, and $b$ all primes greater than $3$. Assume that $b \geq$.  Further assume that $x \equiv a \equiv b \equiv 1$ (mod 5). Then $\frac{11}{3}b < p$.
\label{tighter b bound}
\end{lemma}
\begin{proof}
This follows the same sort of logic as Lemma \ref{2blessthanp}. In Case I we get that $$\frac{k_b b^2}{a+b+1} + \frac{k_b b}{a+b+1} +1 \leq p.$$ Note that congruence arguments give us that $k_b \geq 6$ and thus $a+b+1 \leq 2b+1 < 3b$. We  have by congruence arguments that $k_b \geq 30$, and so  $10b < p$.\\
In Case II, we have as before $\frac{34}{35}x \leq p(b-a)$ and $x+b+1=k_b(b^2+b+1)$. Note that $k_b$ is odd. We have that $x+b+1 \equiv 2$ (mod 3), and $b^2+b+1 \equiv 1$ (mod 3) so $k_b \equiv 2$ (mod 3). Similarly, we have that $k_b \equiv 1$ (mod 5). So $k_b \geq 11$. We have then $$b^2+b+1 \leq \frac{x+b+1}{11}.$$ This is the same as $$11b^2 + 10b + 10 \leq x,$$ and so
$$11b^2 +10b + 10 \leq \frac{35}{34}p(b-a),$$ which is stronger than the desired inequality.\\
In Case III, we have $(12/13)x \leq p(b-a) - b -1.$ and $k_b = \frac{x-b}{b^2+b+1}.$ We have then $30|k_b$, and so
$$b^2 + b + 1 \leq \frac{x-b}{30}$$ and so $30b^2 + 31b +30 \leq x$, and combining as before yields an inequality stronger than the one in equation. \\
In Case IV, we that $k_b = \frac{x+b+1}{b^2+b+1}$  as in Case II. As in Case II, we get that $k_b$ is odd, $k_b \equiv 2$ (mod 3) and $k_b \equiv 1$ (mod 5). So $k_b \geq 11$.  In case IV we had that $k_bb^2 <p(a+b+1)$ and since $a+b+1 < 3b$ this becomes $$11b^2 < 3pb$$ which implies that $$\frac{11}{3}b < p.$$
\end{proof}
\begin{lemma} Suppose that $x,a,b,p$ are all primes greater than 3 where $$x^2+x+1=(a^2+a+1)(b^2+b+1)p.$$  Suppose also that  $a^2+a+1$, $b^2+b+1$ are prime, and that $a \leq b$.  Then in Case II and Case III we have $a^2 + a + 1< \frac{p}{16}$.
\label{Case II and III Lemma}
\end{lemma}
\begin{proof} Assume we are in Case II. So we have $x+a+1|p(b-a) - \frac{x+b+1}{b^2+b+1}$. We note that $x+a+1 \equiv 2$ (mod 3) and $p(b-a) - \frac{x+b+1}{b^2+b+1} \equiv 1$ (mod 3). Note also that both quantities are odd. We thus have
$m(x+a+1)= p(b-a) - \frac{x+b+1}{b^2+b+1}$ for some $m \geq 5$. We have then
$$5(x+a+1) \leq p(b-a) -  \frac{x+b+1}{b^2+b+1}.$$
We have then $x < \frac{pb}{5} -2.$ Then
$$(a^2+a+1)(b^2+b+1)p= x^2+x+1  = (\frac{pb}{5} -2)^2 + \frac{pb}{5} -2  + 1 < \frac{p^2b^2}{25}.$$
So we have
$$(a^2+a+1)(b^2+b+1)p< \frac{p^2b^2}{25}$$
and so we have that $a^2+a +1 < \frac{p}{25}$ which implies the desired result.
Now for Case III: We have that $x+b+1|p(b-a) + \frac{x+a+1}{a^2+a+1}.$  By similar logic as above we have that
$m(x+b+1) = p(b-a) + \frac{x+a+1}{a^2+a+1}$ where $m \equiv 5$ (mod 6). We have
$5(x+b+1) \leq p(b-a) + \frac{x+a+1}{a^2+a+1}$ and so $$x < \frac{pb}{4}-2$$ Then
$$(a^2+a+1)(b^2+b+1)p= x^2+x+1 =  (\frac{pb}{4}-2)^2 +  \frac{pb}{4}-2 +1 < \frac{p^2b^2}{16}, $$
from which the result follows.
\end{proof}
\begin{lemma}Suppose that $x,a,b$ and $p$ are all primes greater than 3 where $$x^2+x+1=(a^2+a+1)(b^2+b+1)p.$$  and $a^2+a+1$, $b^2+b+1$ are prime, and that $a \leq b$. Assume further that we have Case IV and that we have $x-a = p(a+b+1)-k_b$ and $x-b=p(a+b+1)-k_a$. Then we must have $a=b$ and $a^2+a+2|7(p+1).$ \label{Technical lemma for Case IV}
\end{lemma}
\begin{proof} Assume as given. We have then \begin{equation} (x-b)(a^2+a+1)+x+a+1 = p(a+b+1)(a^2+a+1) \label{Case IVa x-b basic}
\end{equation}
and
\begin{equation} (x-a)(b^2+b+1)+x+b+1 = p(a+b+1)(b^2+b+1). \label{Case IVa x-a basic}
\end{equation}
Assume for now that $a \neq b$. from Equations \ref{Case IVa x-b basic} and \ref{Case IVa x-a basic} that
\begin{equation} \frac{(x-b)(a^2+a+1)+x+a+1}{a^2+a+1}= p(a+b+1) = \frac{(x-a)(b^2+b+1)+x+b+1}{b^2+b+1}.
\end{equation}
We can solve the above for $x$ to get that
\begin{equation} x = \frac{(b-a)(a^2+a+1)(b^2+b+1) - (b+1)(a^2+a+1) + (a+1)(b^2+b+1)}{b^2+b-a^2-a} \label{x solve}.
\end{equation}
Now, it turns out the top and bottom both have a factor of $b-a$ and this is a meaningful solution for $x$ because we have assumed that $b \neq a$. Simplifying we get that $$x= \frac{a^2b^2 +a^2b +a^2 + ab^2 +2ab +2a + b^2 +2b +1}{a+b+1},$$ and this forces $x$ to be even which is a contradiction since $x$ is an odd prime. We may thus assume that $a = b$.

Thus, both Equation \ref{Case IVa x-b basic} and Equation \ref{Case IVa x-a basic} become the same thing:
\begin{equation} \label{Case IVa a=b}(x-a)(a^2+a+1)+x+a+1=p(2a+1)(a^2+a+1).
\end{equation}
We may rearrange Equation \ref{Case IVa a=b} to obtain
\begin{equation} x(a^2+a+2)= (p(2a+1)+a)(a^2+a+1) -a-1.
\end{equation}
 \begin{equation}\label{Case IV a non-trivial aa + a+ 2 divis} a^2 + a+2 | (p(2a+1)+a)(a^2+a+1) -a-1 .\end{equation}. We also trivially have
\begin{equation} a^2+ a + 2 | (p(2a+1)+a)(a^2 + a +2) \label{Case IV trivial aa + a + 2 divis}.
\end{equation}
We may then conclude that $a^2+a+2$ divides the difference of \ref{Case IV a non-trivial aa + a+ 2 divis} and \ref{Case IV trivial aa + a + 2 divis}. So $a^2+a+2$ divides $p(2a+1)+2a+1=(p+1)(2a+1)$.
It is easy to check that $gcd(a^2+a+2, 2a+1)|7$. Thus, we have
\begin{equation}\label{Case IVa Strong divisibility of aa + a +2 claim} a^2 + a +2 |7(p+1).
\end{equation}
which was what was claimed.
\end{proof}
\begin{lemma}Suppose that $x,a,b$ and $p$ are all primes greater than 3 where $$x^2+x+1=(a^2+a+1)(b^2+b+1)p.$$ Suppose further that $x\equiv a \equiv b \equiv 1$ (mod 5), and $a^2+a+1$, $b^2+b+1$ are prime, and that $a \leq b$, and $p\geq 47$.  Then we have $a^2 + a + 1< p$.
\label{Type I aa + a +1 < p }
\end{lemma}
\begin{proof}
 We will break into four cases in the same way as before.  Cases II and III are handled by Lemma \ref{Case II and III Lemma} In this situation, almost all the serious work will be in Case IV.
In Case I:  We have as before that $x+b+1|p(a+b+1) +k_a$ where $k_a = \frac{x-a}{a^2+a+1}$. Set $m(x+b+1)=p(a+b+1) +k_a$.  We have that $x+b+1 \equiv 3$ (mod 5)
and that $p(a+b+1) +k_a \equiv 1$ (mod 5). Thus, we have $m \equiv 2$ (mod 5). Similarly, we have that $m \equiv 1$ (mod 3). So $m \geq 7$ and so
$7(x+b+1) \leq p(a+b+1)+k_a$. Since $a^2+a+1 >7$, we have that $k_a \leq \frac{x}{7}$,
so $$\frac{48}{7}x \leq p(a+b+1) - 7b -7$$ which yields
$$x \leq \frac{7}{48}p(a+b+1) - \frac{49}{48}b - \frac{49}{48}.$$ We have that $a+b+1 <3b$ and so
$$x \leq \frac{7}{16}pb - \frac{49}{48} - \frac{49}{48} \leq \frac{7}{16}pb -2.$$
We have then
$$(a^2+a+1)(b^2+b+1)p=x^2+x+1 \leq (\frac{7}{16}pb -2)^2 + \frac{7}{16}pb -2 + 1 < \left(\frac{7}{16}\right)^2(pb^2). $$ From the above we get that
$a^2+a+1 \leq \left(\frac{7}{16}\right)^2p < p. $
 Now for Case IV. We have  $(x-a)m_a = p(a+b+1)-k_b$ and $(x-b)m_b=p(a+b+1)-k_a$ for some $m_a$ and $m_b$. In the first case, assume that $m_a = m_b=1$.
 
Then  by Lemma \ref{Technical lemma for Case IV} we have that  $k(a^2+a+2)=7(p+1)$ for some $k$. Note that $a^2+a+2 \equiv 4$ (mod 5). We also have $7(p+1) \equiv 1$ (mod 5). Thus, we have $k \equiv 4$ (mod 5). Similarly, we have that $k \equiv 1$ (mod 3). So $k \equiv 4$ (mod 15). Consider the possibility of $k=4$. If that is the case then we have
$4(a^2+a+2)=7(p+1)$, and this is the same as
$$7p = 4a^2 +4a +1 = (2a+1)^2.$$ But then we must have $p=7$ and $a=3$, which is ruled out by our initial assumption that $a>3$. Thus, since $k \equiv 4$ (mod 15), we get that $k \geq 19$. In that case we have $19(a^2+a +1) \leq 7(p+1) $, from which it follows that $a^2 + a + 1 <p.$
Next we will consider $m_a=2$ and $m_b=1$. We need to consider then the equation $$2(x-a) = p(a+b+1) -k_b.$$ Since $m_b \neq m_a$ we have that $a \neq b$, and thus $a \leq b -30$ (since $a$ and $b$ agree (mod 30)). So we have that
$$x = \frac{p(a+b+1)}{2} - \frac{k_b}{2} + a \leq  \frac{p(2b-29)}{2} - \frac{k_a}{2} +a. $$
We note that Lemma \ref{2blessthanp} implies that $a <p$, and so we have that
$$x \leq pb - \frac{29p}{2} + p  < pb -2,$$ from which the same logic as used in Case I holds.
The case of $m_b=2$ and $m_a=1$ is nearly identical, as is the case of either $m_a$ or $m_b$ being 2 but with $a \neq b$.
We will then next consider the case $m_a = m_b=2$, and $a=b$.
We have $$2(x-a)=p(2a+1) - \frac{x+a+1}{a^2+a+1}.$$ Solving for $x$ we obtain:
$$(2a^2 + 2a +3)x = (p(2a+1) + 2a)(a^2+a+1) - a-1.$$ Thus, we have
$2a^2 +2a +3|(p(2a+1) + 2a)(a^2+a+1) - a-1$, and we get that \begin{equation}\label{2aa + 2a +3 nontrivial division}2a^2 +2a +3|(p(2a+1) + 2a)(2a^2+2a+2) - 2a-2.\end{equation} Trivially,
\begin{equation}\label{2aa +2a +3 trivial division} 2a^2 +2a +3|(p(2a+1)+2a)(2a^2+2a+3).
\end{equation}
Taking the difference between the quantities on the right-hand sides of Equation \ref{2aa + 2a +3 nontrivial division} and Equation \ref{2aa +2a +3 trivial division}, we get that $$2a^2 + 2a+3|p(2a+1)+4a +2.$$ We note that both quantities in the above agree mod 10. So there is some $m \equiv 1$ (mod 10) such that
$$m(2a^2+2a+3) = p(2a+1) + 4a +2.$$ We claim that $m \neq 1$. To see this, assume that $m=1$, so we have
$$2a^2 +2a +3 = p(2a+1) + 4a +2,$$ which is the same as $$2a^2 -2a +1 =p(2a+1).$$ Thus we have $2a+1|2a^2 -2a +1$. It is easy to check that the only positive integer which satisfies this is $a=2$. This is impossible since $a$ is an odd prime.
By the above remarks we must have either $m_a \geq 3$ or $m_b \geq 3$. We will consider $m_a \geq 3$ (the case for $m_a \geq 3$ is essentially identical).
We have $$3(x-a) \leq p(a+b+1) - k_b.$$ We note that $a+b+1 \leq 2b+1$, and thus
$$x \leq \frac{p(2b+1)}{3} - k_b + a.$$
We note that $p \geq 47$, together with Lemma \ref{49over4p} gives us that $a < 1/3p$. We again obtain
that $$x \leq pb -2,$$ and again draw the same conclusion.
\end{proof}
\begin{lemma}Suppose that we have odd primes  $x,a,b,p$ all primes greater than 3 where $$x^2+x+1=(a^2+a+1)(b^2+b+1)p.$$ Suppose further that $x\equiv a \equiv  1$ (mod 5), $b \equiv 2$ (mod 5) and $a^2+a+1$, $b^2+b+1$ are prime, and that $a \leq b$, and $p\geq 47$.  Then we have $a^2 + a + 1< p$.
\label{Type II aa + a +1 <p}
\end{lemma}
\begin{proof} As usual, we will have four cases to check, and as in the last lemma most of the difficulty will be in Case IV. In this lemma as in the last one we will be able to rely on the previous results concerning this situation, but note that due to the different congruence assumption we cannot here make use of Lemma \ref{Type I aa + a +1 < p }.
Note that in this lemma we now have $p \equiv 3$ (mod 5), rather than in the previous lemma where we had $p \equiv 2$ (mod 5).
Case I:  We have as before that $x+b+1|p(a+b+1) + \frac{x-a}{a^2+a+1}$. Set $m(x+b+1)=p(a+b+1) +\frac{x-a}{a^2+a+1}$.
We have that $p(a+b+1) + \frac{x-a}{a^2+a+1} \equiv 2$ (mod 5) and $x+b+1 \equiv 4$ (mod 5). Thus $m \equiv 3$ (mod 5). Note that $m$ is odd, and also that $(3,m)=1$, and so $m \geq 13$. We have then
$$13(x+b+1) \leq p(a+b+1) + \frac{x-a}{a^2+a+1}.$$ The same logic as in the previous lemma then applies.
Lemma \ref{Case II and III Lemma} handles Case II and Case III.  \\
Case IV: We have $a^2+a+1|x+a+1$ and $b^2+b+1|x+b+1$. Note that unlike in Lemma \ref{Type I aa + a +1 < p } we have immediately that $a \neq b$ because they disagree mod $5$. We have $x-a|p(a+b+1)-k_b$ and $x-b|p(a+b+1)-k_a$ where $k_b= \frac{x+b+1}{b^2+b+1}$, and $k_a = \frac{x+a+1}{a^2+a+1}.$
We have $(x-a)m_a = p(a+b+1)-k_b$ and $(x-b)m_b=p(a+b+1)-k_a$ for some positive integers $m_a$ and $m_b$. We will consider various possible options for the pair $(m_a,m_b)$. The pair $(1,1)$ is already ruled out since then Lemma \ref{Technical lemma for Case IV} would force  $a=b$.
 
Since we have that $a \neq b$, and thus $a \leq b -6$ (since $a< b$ and $a \equiv b+1$ (mod 30)). So we have that
$$x = \frac{p(a+b+1)}{2} - \frac{k_b}{2} + a \leq  \frac{p(2b-6)}{2} - \frac{k_a}{2} +a. $$
Note that Lemma \ref{2blessthanp} implies that $a <p$, and so we have that
$$x \leq pb - 3p + p  < pb -2,$$ from which the same logic as used in Case I holds.
For the remaining possible options for $(m_a,m_b)$ we   follow logic essentially identically to those in remaining parts of the Case IV of the proof of Lemma \ref{Type I aa + a +1 < p }.
\end{proof}
Using nearly identical logic to the above lemma we have:
\begin{lemma}Suppose that we have odd primes  $x,a,b,p$ all primes greater than 3 where
$$x^2+x+1=(a^2+a+1)(b^2+b+1)p.$$ Suppose further that $ a \equiv  2$ (mod 5), $x \equiv b \equiv 1$ (mod 5). Assume that $a^2+a+1$ and $b^2+b+1$ are prime. Finally, assume that $a \leq b $, and $p\geq 47$.  Then $a^2 + a + 1< p$.
 \label{Type III aa + a +1 <p}
 \end{lemma}

\begin{lemma} There are no solutions to the equation $$x^2+x+1=(p^2+p+1)(q^2+q+1)(r^2+r+1)$$ with $x$, $p$, $q$, $r$, $p^2+p+1$, $q^2+q+1$, $r^2+r+1$ all prime and with $x \equiv p \equiv q \equiv 1$ (mod 5). That is, there does not exist any triple threat $(x,p,q,r)$ where $x \equiv p \equiv q \equiv 1$ (mod 5).
\label{No triple threats with x, a, b all 1 mod 5}
\end{lemma}
\begin{proof} Assume we have a solution. We note that we must also have $r \equiv 2$ (mod 5). We must have $\min(p^2+p+1, q^2+q+1,q^2+q+1) > 47$. If any were not, we could use  Corollary \ref{Maximum obstruction size} to conclude that we have $47> x^{2/5}$,  this gives us only a finite set of $x$ to check and we can easily verify that none of them are solutions. We may without loss of generality also assume that $p \leq q$. We  conclude by Lemma \ref{Type I aa + a +1 < p } that $p^2+p+1 < r^2+r+1$, and so $p<r$. We may apply Lemma  \ref{Type II aa + a +1 <p}  to get that  $p<q$. We have two cases to consider. Either $p<r<q$ or $p<q<r$. If $p<r<q$, then Lemma \ref{Type II aa + a +1 <p} gives us that $r^2+r+1 < p^2+p+1$ and hence $r<p$ which is impossible.
So we may assume that $p<q<r$, but then by Lemma \ref{Type III aa + a +1 <p} we have that $q^2+q+1<p^2+p+1$ and hence $q<p$ which is impossible. So all cases have lead to a contradiction.
\end{proof}
The basic thrust of the next set of results is very similar.
\begin{lemma} Suppose that $x^2+x+1=(a^2+a+1)(b^2+b+1)p$ where $x,a,b,a^2+a+1, b^2+b+1$ and $p$ are odd primes greater than 3. Suppose further that $p=c^2+c+1$ for some $c \equiv 5$ (mod 6) . Suppose that $a \leq b$ and that we are in Case I. Then $a^2+a+1 < \frac{p}{4}$.
\label{Case I Lemma assuming that p=c^2+c+1}
\end{lemma}
\begin{proof} Assume as given. So we have $x+a+1|p(a+b+1) + \frac{x-b}{b^2+b+1}$ and $x+b+1|p(a+b+1) + \frac{x-a}{a^2+a+1}$.
We may set $m_a(x+a+1)=p(a+b+1) + \frac{x-b}{b^2+b+1}$ and $m_b(x+b+1)= x+b+1|p(a+b+1) + \frac{x-a}{a^2+a+1}$
We'll first assume that $m_a=m_b=1$ and then handle the remaining cases.
If $m_a=m_b=1$ then we have
\begin{equation}\label{Case I m_a=1} x+a+1=p(a+b+1) + \frac{x-b}{b^2+b+1}\end{equation} and \begin{equation}\label{Case I m_b=1} x+b+1= p(a+b+1) + \frac{x-a}{a^2+a+1}\end{equation}
We'll first assume that $a \neq b$, and arrive at a contradiction. We'll then handle $a=b$ (which is where we will need the assumption that $p=c^2+c+1$).
Assume that $a \neq b$. Then we may subtract Equation \ref{Case I m_a=1} from Equation \ref{Case I m_b=1} to get that:
$$b-a = \frac{x-a}{a^2+a+1} - \frac{x-b}{b^2+b+1}.$$ This is the same as
$$(a^2+a+1)(b^2+b+1)(b-a)= x(b-a)(a+b+1) + ab(b-a) + (b-a). $$
Since $b \neq a$ we have $b - a \neq 0$ and so we may divide by $b-a$ to get
$$(a^2+a+1)(b^2+b+1)=x(a+b+1) + ab+1. $$ This is the same as
$x= \frac{a^2b^2 +a^2b + b^2a + a^2 +b^2 + a+b}{a+b+1}$. But we must have $a \equiv b \equiv 2$ (mod 3) and also $x \equiv 2$ (mod 3). But $a \equiv b \equiv 2$ (mod 3) forces the right hand-side of the above to be $1$ (mod 3). So we must have $a=b.$
Since $a=b$ we have
 $$x+a+1 = p(2a+1) + \frac{x-a}{a^2+a+1}.$$
which can be rearranged to $$(p-1)(2a^2 +3a^2 +3a+1)= -a^3 + a^2(x-a) + ax.$$
We have that $(a, 2a^2 +3a^2 +3a+1)=1$ and so $a|p-1$. Now since $p=c^2+c+1$ this is the same as saying that $a|c(c+1)$. Either $a|c$ or $a|c+1$. If $a|c$ then either $a=c$, or $a<c$. If $a=c$ and so $p=a^2+a+1$ and $x^2+x+1=(a^2+a+1)^3$. But this would contradict Lemma \ref{x2+x+1 not a cube of a prime}. If $a<c$, then since $a|c$ and $a \equiv c \equiv 5$ (mod 6), so we would have $7a \leq c$, from which it easily follows that $a^2+a+1 < \frac{p}{4}$. If we have $a|c+1$, then since $6|c+1$ and we have that $6a|c+1$ and so $6a \leq c+1$, from which it easily follows that $a^2+a+1 <\frac{p}{4}$.
We now need to handle the case of $m_a$ and $m_b$ are not both equal to 1. We will look at the case of $m_a \neq 1$ ($m_b \neq 1$ is essentially identical).
We have that $m_a(x+a+1)=p(a+b+1) + \frac{x-b}{b^2+b+1}$ for some $m_a>1$. We note that we cannot have $m_a$ even because the right-hand side of the equation is odd. We also cannot have $m_a=3$ because the right hand-side is $2$ (mod 3). We therefore have $m_a \geq 5$.
We have then:
$5(a+x+1) \leq p(a+b+1) + \frac{x-b}{b^2+b+1}$ from which the desired inequality follows.
\end{proof}
We would like in Lemma \ref{Case I Lemma assuming that p=c^2+c+1} to remove the need for the assumption that $p=c^2+c+1$ but for our results here at least that is not necessary.
\begin{lemma} Suppose that $x^2+x+1=(a^2+a+1)(b^2+b+1)p$ where $x,a,b,p,a^2+a+1$ and $b^2+b+1$ are all prime. Assume that $x \equiv a \equiv 1$ (mod 5), and that $b \equiv 4$ (mod 5), and that $a < b$. Assume also that $p=c^2+c+1$ for some $c \equiv 5$ (mod 6). Then $a^2+a+1 < \frac{p}{4} $ \label{1,1,1,4 first lemma}
\end{lemma}
\begin{proof} Note that $p \equiv 1$ (mod 5). We again split into four cases. Case I is handled by Lemma \ref{Case I Lemma assuming that p=c^2+c+1}.   As usual, Cases II and III are handled by Lemma \ref{Case II and III Lemma}.  So we need only concern ourselves with Case IV.
In Case IV we have:  $x-b|p(a+b+1)-\frac{x+a+1}{a^2+a+1}$. We have then for some  $m_b$ that
 $$m_b(x-b)= p(a+b+1)-\frac{x+a+1}{a^2+a+1}.$$ Note $x-b \equiv 2$ (mod 5), and $p(a+b+1) - \frac{x+a+1}{a^2+a+1} \equiv 0$ (mod 5). So $m_b \geq 5$
We have then that $5(x-b)\leq p(a+b+1)-\frac{x+a+1}{a^2+a+1}$ from which the desired bound follows.
\end{proof}
Using nearly identical logic we have:
\begin{lemma}\label{1,1,4,4 second lemma} Suppose that $x^2+x+1=(a^2+a+1)(b^2+b+1)p$ where $x,a,b,p,a^2+a+1$ and $b^2+b+1$ are all prime. Assume that $x \equiv b \equiv 1$ (mod 5),  $a \equiv 4$ (mod 5), and  $a < b$. Assume also that $p=c^2+c+1$ for some $c \equiv 5$ (mod 6). Then $a^2+a+1 < \frac{p}{4} $.
\end{lemma}
\begin{lemma}\label{1,1,4,4 third lemma}  Suppose that $x^2+x+1=(a^2+a+1)(b^2+b+1)p$ where $x,a,b,p,a^2+a+1$ and $b^2+b+1$ are all prime. Assume that $x \equiv 1$ (mod 5), that $a \equiv b \equiv 4$ (mod 5) and that $p \equiv 3$ (mod 5), and that $a < b$. Then $a^2+a+1 <\frac{p}{16}$.
\end{lemma}
\begin{proof} We have our four cases as usual with Cases II and III handled by Lemma \ref{Case II and III Lemma}. In Case I we have  We have as before that $x+b+1|p(a+b+1) +\frac{x-a}{a^2+a+1}$ Set $m(x+b+1)=p(a+b+1) +\frac{x-a}{a^2+a+1}$.  We have that $x+b+1 \equiv 1$ (mod 5) and $p(a+b+1) +\frac{x-a}{a^2+a+1} = 4$ (mod 5). So $m \equiv 4$ (mod 5). Since $(6,m)=1$ we  have that $m \geq 19$. We have then that  $19(x+b+1) \leq p(a+b+1) +\frac{x-a}{a^2+a+1}$ from which the desired inequality follows. \\
In Case IV we have   $x-b|p(a+b+1)-\frac{x+a+1}{a^2+a+1}$. We set $m(x-b)=p(a+b+1)-\frac{x+a+1}{a^2+a+1}$. We note that $x-b \equiv 2$ (mod 5), and $p(a+b+1)-\frac{x+a+1}{a^2+a+1} \equiv 1 $ (mod 5). Thus,  $m \equiv 2$ (mod 5) and so $m \geq 7$. We have then that
$7(x-b) \leq p(a+b+1) - \frac{x+a+1}{a^2+a+1}$ from which the desired inequality follows.
\end{proof}
\begin{lemma}\label{No triple threat of form (1,1,4,4)} There are no solutions to the equation $$x^2+x+1=(p^2+p+1)(q^2+q+1)(r^2+r+1)$$ with $x$, $p$, $q$, $r$, $p^2+p+1$, $q^2+q+1$, $r^2+r+1$ all prime and with $x \equiv p \equiv 1$ (mod 5) and $q \equiv  r \equiv 4$ (mod 5).  That is, there does not exist any triple threat $(x,p,q,r)$ where $x \equiv p \equiv 1$ (mod 5) and $q \equiv r \equiv 4$ (mod 5).
\end{lemma}
\begin{proof} Assume we have such a solution. Without loss of generality we may assume that $q \leq r$. We have then from Lemma \ref{1,1,4,4 third lemma} that $q < p$. We thus have either $q < p < r$ or $q < r < p$ (we cannot have $p=r$ since they disagree mod 5). If we have that $q< p< r$, then from Lemma \ref{1,1,1,4 first lemma} we have that $p<q$ which is a contradiction.  If we have $q < r < p$ then by Lemma \ref{1,1,4,4 second lemma} we have that $r < q$ which is a contradiction. Since all possibilities lead to a contradiction, the corresponding type of triple threat cannot exist.
\end{proof}
We now turn our attention to our final type of triple threat.
\begin{lemma}\label{(1,1,2,3) first lemma}Suppose that $x^2+x+1 = (a^2+a+1)(b^2+b+1)p$ where $x,a,b, a^2+a+1, b^2+b+1$ and $p$ are primes. Suppose further that $x \equiv b \equiv 1$ (mod 5), $a \equiv 3$ (mod 5), and $p \equiv 2$ (mod 5). Suppose that $a < b$. Then we have that $a^2 + a +1 < p$.
\end{lemma}
\begin{proof} We split into four cases as usual with cases II and III handled by lemma \ref{Case II and III Lemma}. In Case I we have that
$x+b+1|p(a+b+1) + \frac{x-b}{b^2+b+1}$. We have that $m(x+b+1)=p(a+b+1) + \frac{x-b}{b^2+b+1}$ for some $m$. We note that $x+b+1 \equiv 3$ (mod 5) and
$p(a+b+1) + \frac{x-b}{b^2+b+1} \equiv  1$ (mod 5). We have then that $m \equiv 2$ (mod 5). Since $m$ is odd, we have that $m \geq 7$,
and $7(x+b+1) \leq p(a+b+1) + \frac{x-b}{b^2+b+1}$, from which the desired inequality follows.
In Case IV we have that $x-a|p(a+b+1) - \frac{x+b+1}{b^2+b+1}$. We have that for some $m$, $m(x-a) = p(a+b+1) - \frac{x+b+1}{b^2+b+1}$. We have that $x-a \equiv 3$ (mod 5), and $p(a+b+1) - \frac{x+b+1}{b^2+b+1} \equiv 4$ (mod 5). We have then that $m \equiv 3$ (mod 5). We have then that
$3(x-b) \leq p(a+b+1) - \frac{x+b+1}{b^2+b+1}$ from which the inequality follows.
\end{proof}
Using nearly identical logic we have:
\begin{lemma}\label{(1,1,2,3 second lemma} Suppose that $x^2+x+1 = (a^2+a+1)(b^2+b+1)p$ where $x,a,b, a^2+a+1, b^2+b+1$ and $p$ are primes. Suppose further that $x \equiv a \equiv 1$ (mod 5), $b \equiv 3$ (mod 5), and $p \equiv 2$ (mod 5). Suppose that $a < b$. Then we have that $a^2 + a +1 < p$.
\end{lemma}
\begin{lemma}\label{1,1,2,3 third lemma} Suppose that $x^2+x+1 = (a^2+a+1)(b^2+b+1)p$ where $x,a,b, a^2+a+1, b^2+b+1$ and $p$ are primes. Suppose further that $x \equiv 1$ (mod 5), $a \equiv 2$ (mod 5), $b \equiv 3$ (mod 5), and $p \equiv 3$ (mod 5). Suppose that $a < b$. Then we have that $a^2 + a +1 < p$.
\end{lemma}
\begin{proof} We again split into four cases and handle cases II and III via lemma \ref{Case II and III Lemma}. In Case I, we have that
$x+a+1|p(a+b+1) + \frac{x-b}{b^2+b+1}$. We have that $x+a+1 \equiv 4$ (mod 5), and $p(a+b+1) + \frac{x-b}{b^2+b+1} \equiv 1$ (mod 5). The rest of the case follows as usual.
For Case IV we have that $x-a|p(a+b+1) - \frac{x+b+1}{b^2+b+1}$. We have that $x-a \equiv 4$ (mod 5), and that $p(a+b+1) - \frac{x+b+1}{b^2+b+1} \equiv 3$ (mod 5), and the rest of the argument follows as usual.
\end{proof}
By nearly identical logic we have:
\begin{lemma}\label{1,1,2,3 fourth lemma} Suppose that $x^2+x+1 = (a^2+a+1)(b^2+b+1)p$ where $x,a,b, a^2+a+1, b^2+b+1$ and $p$ are primes. Suppose further that $x \equiv 1$ (mod 5), $a \equiv 3$ (mod 5), $b \equiv 1$ (mod 5), and $p \equiv 2$ (mod 5). Suppose that $a < b$. Then we have that $a^2 + a +1 < p$.
\end{lemma}
\begin{proof} We again split into four cases and handle cases II and III via Lemma \ref{Case II and III Lemma}. In Case I, we have that
$x+a+1|p(a+b+1) + \frac{x-b}{b^2+b+1}$. We have that $x+a+1 \equiv 4$ (mod 5), and $p(a+b+1) + \frac{x-b}{b^2+b+1} \equiv 1$ (mod 5). The rest of the case follows as usual.\\
For Case IV we have that $x-a|p(a+b+1) - \frac{x+b+1}{b^2+b+1}$. We have that $x-a \equiv 4$ (mod 5) and that $p(a+b+1) - \frac{x+b+1}{b^2+b+1} \equiv 3$ (mod 5). The rest of the argument follows as usual.
\end{proof}

We are now in a position where we may prove:
\begin{lemma}\label{No triple threat of form (1,1,2,3)}   There are no solutions to the equation $$x^2+x+1=(p^2+p+1)(q^2+q+1)(r^2+r+1)$$ with $x$, $p$, $q$, $r$, $p^2+p+1$, $q^2+q+1$, $r^2+r+1$ all prime and with $x \equiv p equiv 1$ (mod 5) and $q \equiv 2$ (mod 5), and $r \equiv 3$ (mod 5).  That is, there does not exist any triple threat $(x,p,q,r)$ where $x \equiv p \equiv 1$ (mod 5) and $q \equiv 2$ (mod 5), $r \equiv 3$ (mod 5).
\end{lemma}
\begin{proof} Assume we have such an $(x,p,q,r)$. Either $p<q$ or $q<p$ (they cannot be equal since they disagree mod 5). First, let us consider the case that $p < q$. Then by Lemma \ref{Type II aa + a +1 <p} we have that $p < r$. We have then either $p < q < r$ or $p < r< q$.  Let us first consider the case of $p < q< r. $ We have then by Lemma \ref{1,1,2,3 third lemma} that $q <p$ which is a contradiction. Let us then consider the case $p < r < q$.  We may then use Lemma \ref{1,1,2,3 fourth lemma} to conclude that $p< r$ which is a contradiction.  Thus, both of the possibilities for $p<q$ lead to a contradiction. We thus must have $q< p$. From $q< p$ and  Lemma \ref{Type III aa + a +1 <p} we me must have $q < r$. Thus we have either $q < p < r$ or $q < r < p$. If we have $q < p < r$, then by Lemma \ref{1,1,2,3 fourth lemma} we have that $p< q$. If  $q < r < p$ we may use Lemma  \ref{1,1,2,3 fourth lemma}  to get that $r< q$ and so we have a contradiction. So in each situation we have a contradiction and so the intended type of triple threat does not exist.
\end{proof}

\begin{lemma} If  $(x,a,b,c)$ is a triple threat with $x\equiv 1$ (mod 5) then none of $a,b$ or $c$ may be $1$ (mod 5).
\label{No triple threats with x = 1 mod 5 and any of a b or c 1 mod 5}
\end{lemma}
\begin{proof} We can enumerate all possible triple threats mod 5 (up to order of the variables) every triple threat where $x \equiv 1$ (mod 5) and at least one of $a$, $b$ or $c$ is $1$ (mod 5). We then see  that the triple threat must be one of a form ruled out by  Lemma \ref{No triple threats with x, a, b all 1 mod 5}, Lemma \ref{No triple threat of form (1,1,4,4)}, or Lemma \ref{No triple threat of form (1,1,2,3)}.
\end{proof}
\begin{lemma} \label{Precursor Lemma} Suppose that $a$ and $c$ are distinct  odd primes, with $a^2+a+1$ prime. Assume further that $a^2+a+1|c^2+c+1.$  Then  $a^2+a+1 < \frac{c}{2}$. Furthermore, if $(3,c^2+c+1)=1$, then $a^2 + a + 1 < \frac{2}{9}c$.
\end{lemma}
\begin{proof} Assume as given. Since $c^2+c+1 \equiv 0$ (mod $a^2+a+1$), and $a^2+a+1$ is prime, we must have either $c \equiv a$ (mod $a^2+a+1)$ or $c \equiv a^2$ (mod $a^2+a+1$) (since $a^2+a+1|(c-a)(a+c+1)$). We  have $c \neq a$ by assumption. We also have that $c \neq a^2$ since $c$ is prime. We then note that $a+(a^2+a+1)$ and $a^2+(a^2+a+1)$ are both even and so $c$ cannot be equal to either. Thus, we have that $$c \geq a+2(a^2+a+1) > 2(a^2+a+1)$$ from which the result follows.

Now, under the additional assumption that $(3,c^2+c+1)=1$, we must have either $c=3$ (which immediately leads to a contradiction), or we must have $c \equiv 2$ (mod 3). Since $a^2+a+1$ is prime, we must have $a \equiv 2$ (mod 3). Because $a+2(a^2+a+1) \equiv 1$ (mod 3), we have $c \neq a+2(a^2+a+1)$. Similarly, $a^2+2(a^2+a+1) \equiv 0$ (mod 3), so $c \neq a^2 + 2(a^2+a+1)$. We can rule out the next two possible values for $c$, $a+3(a^2+a+1)$ and $a^2+3(a^2+a+1)$ since they are both even. Then $a + 4(a^2+a+1) \equiv 0$ (mod 3), and this is not an acceptable value of $c$ either, and so $$c \geq a^2+4(a^2+a+1) > \frac{9}{2}(a^2+a+1)$$ which is the desired inequality.
\end{proof}
\begin{lemma} There are no odd primes $a,b,c$, with $a^2+a+1$ and $b^2+b+1$ prime, and satisfying
$$c^2+c+1=3(a^2+a+1)(b^2+b+1).$$\label{Other new lemma for repairing 3 divides N case}
\end{lemma}
\begin{proof}
Assume we have odd primes  $a,b,c$, with both $a^2+a+1$ and $b^2+b+1$ prime, and satisfying
$$c^2+c+1=3(a^2+a+1)(b^2+b+1).$$
We may then apply Lemma \ref{Precursor Lemma} twice to conclude that $$a^2+a+1 < \frac{c}{2}$$ and that $$b^2+b+1 < \frac{c}{2}. $$
We then have $$c^2+c+1=3(a^2+a+1)(b^2+b+1)< 3\left(\frac{c}{2}\right)\left(\frac{c}{2}\right) = \frac{3c^2}{4}  < c^2+c+1$$ which is a contradiction.
\end{proof}

\begin{lemma} Suppose that $a$ and $c$ are distinct odd primes. Assume further that $\frac{a^2+a+1}{3}$ is prime and that $\frac{a^2+a+1}{3}|(c^2+c+1)$. Then
either $a^2 + a +1 < \frac{3}{4}c$ or $c=\frac{a^2+a+1}{3}-(a+1)$.
\label{New lemma make referee happy}
\end{lemma}
\begin{proof} Suppose that $a$ and $c$ are distinct odd primes. Assume further that $\frac{a^2+a+1}{3}$ is prime and that $\frac{a^2+a+1}{3}|(c^2+c+1)$.  Set $p=\frac{a^2+a+1}{3}$. Consider possible $t$ such that $t^2 + t + 1\equiv 0$ (mod $p$). Mod $p$, there are two residue classes which are solutions, $t_1$ and $t_2$. Note if $t_1$ is the smaller solution, then they are related by $t_2=p-t_1-1$. Thus, we must have $t_2> \frac{p}{2}$. We cannot have $a \equiv t_2$ because then we would have $$a^2+a+1 > \left(\frac{p}{2}\right)^2 = \left(\frac{a^2+a+1}{3}\right)^2 > \frac{a^4}{9}.$$ This is a contradiction since we must have $a \geq 7$, and so $a^2 + a +1 > \frac{a^4}{9}$ is impossible. Thus, $a$ must be the smallest positive integer such that $t^2 + t + 1 \equiv 0$ (mod $p$).

Now, consider two cases, where $c \equiv a$ (mod $p$) or where $c \not \equiv a$ (mod $p$).  In the first case, $c \equiv a$ (mod $p$), we cannot have $c =a$ by assumption, and $c+p$ is even, so the next option is $c = 2p+a$ we must have $c \geq 2p+a. $ but $p \equiv 1$ (mod 3) and $a \equiv 1$ (mod 3), so $2p+a \equiv 0$ (mod 3). Thus, this is not a valid option for $c$. Our next choice is $3p+a$, which is even, and so we then have $c \geq 4p+a$, which implies the desired inequality.

Now, consider if $c$ is in the other residue class. This means that $c \equiv a^2 \equiv -(a+1)$ (mod $p$). Consider the simplest case, $c = p-a-1.$ Then $c^2 + c +1 = (p-(a+1))^2 + p-(a +1) +1 = \frac{a^2+a+1}{3}-(a+1)$, as required by the Lemma.

If $c \neq p-a-1$, then we may by the same sort of logic as earlier rule out other small values of $c$. We can rule out $c=2p-a-1$ since this number is even. The next case is $$c=3p-a-1= 3(\frac{a^2+a+1}{3}) - (a+1) = a^2-1=(a-1)(a+1),$$ which is impossible since $c$ is prime.  After that, our next possibility is $c=4p-a-1$ which is even, so we must have $c \geq 5p-a-1$ which implies that $c> 4(a^2+a+1).$


\end{proof}

\begin{lemma}\label{Get no overlap between s1 and s21} There are no solutions to $x^2+x+1=3(a^2+a+1)$ where $a,x$ and $a^2+a+1$ are all odd primes.
\end{lemma}
\begin{proof} Assume that $x^2+x+1=3(a^2+a+1)$ where $a,x$ and $a^2+a+1$ are all odd primes. By Lemma \ref{Precursor Lemma}, we have $a^2+a +1 < \frac{x}{2}$. Thus, $x> 2(a^2+a+1)$, and hence $$3(a^2+a+)1= x^2 + x + 1> x^2 > 4(a^2+a+1), $$ which is a contradiction.

\end{proof}

\begin{lemma}\label{Only partial overlap with s1 and s22} There are  no solutions to
$x^2+x+1=(a^2+a+1)(b^2+b+1)$ where $x,a,b$ $a^2+a+1$, and $b^2+b+1$ are all odd primes.
\end{lemma}
\begin{proof} The proof is essentially identical to that for Lemma \ref{Other new lemma for repairing 3 divides N case}.

\end{proof}

\begin{lemma}\label{No s22 is an s1 and an s21} There are no solutions to $x^2+x+1=(a^2+a+1)(\frac{b^2+b+1}{3})$ where $x,a,b$ $a^2+a+1$ and $\frac{b^2+b+1}{3}$ are all odd primes.
\end{lemma}
\begin{proof} Assume that $x^2+x+1=(a^2+a+1)(\frac{b^2+b+1}{3})$ where $x,a,b$ $a^2+a+1$ and $\frac{b^2+b+1}{3}$ are all odd primes. Then $$3(x^2+x+1)=(a^2+a+1)(b^2+b+1).$$ Since $a^2+a+1$ is a prime greater than $3$ we have $a^2+a+1|x^2+x+1$. Note that we also have that $(3,x^2+x+1)=1$ and hence we may apply the second inequality from Lemma \ref{Precursor Lemma}, to conclude that \begin{equation}\label{aa + a +1 < (2/9)x} a^2 +a +1 < \frac{2}{9}x \end{equation} We may also apply Lemma \ref{New lemma make referee happy} to get either $x=\frac{b^2+b+1}{3}-(b+1)$ or $b^2+b+1 < \frac{3}{4}x$. The second of these would immediately lead to a contradiction with Inequality  \ref{aa + a +1 < (2/9)x}, so we may assume that $x=\frac{b^2+b+1}{3}-(b+1).$ Therefore,
$$x^2+x+1=\left(\frac{b^2 -5b +7}{3}\right)\left(\frac{b^2+b+1}{3}\right).$$
We then must have $\frac{b^2 -5b +7}{3}= a^2+a+1$, and that forces that  
$$x=\frac{b^2+b+1}{3}-(b+1) = \frac{b^2 -5b +7}{3} + b-3=a^2+a+1 + b -3.$$ This contradicts Inequality \ref{aa + a +1 < (2/9)x}, and so we have our final contradiction.
\end{proof}

\begin{lemma}\label{No s22 is two s21s} There are no solutions to $x^2+x+1=(\frac{a^2+a+1}{3})(\frac{b^2+b+1}{3})$ where $x,a,b,$ $\frac{a^2+a+1}{3}$ and $\frac{b^2+b+1}{3}$ are all odd primes.
\end{lemma}
\begin{proof} Assume we have $x^2+x+1=(\frac{a^2+a+1}{3})(\frac{b^2+b+1}{3})$ where $x,a,b,$ $\frac{a^2+a+1}{3}$ and $\frac{b^2+b+1}{3}$ are all odd primes. Note that we cannot have $a = b$ since $x^2+x+1$ cannot be a perfect square.

We now invoke Lemma \ref{New lemma make referee happy}, and we have four cases depending on which of the prongs of Lemma \ref{New lemma make referee happy} is active for $a$ and $b$. In Case I, we have $a^2 +a +1 < \frac{3}{4}x$ and $b^2 + b + 1 < \frac{3}{4}x$. In Case II, we have we have $x=\frac{a^2+a+1}{3}-(a+1)$ and $b^2 + b +1 < \frac{3}{4}x$. In Case III, we have $a^2 +a  +1 < \frac{3}{4}x$ and $x=\frac{b^2+b+1}{3}-(b+1)$. Finally, in Case IV, we have
$x = \frac{a^2+a+1}{3}-(a+1)$ and $x=\frac{b^2+b+1}{3}-(b+1)$. \\

Case I: We obtain a contradiction since we have $$x^2+x+1=\left(\frac{a^2+a+1}{3}\right)\left(\frac{b^2+b+1}{3}\right) < \left(\frac{x}{4}\right).$$

Cases II and III are essentially identical so we will only discuss case II.  We have  $x=\frac{a^2+a+1}{3}-(a+1)$ and $b^2 + b +1 < \frac{3}{4}x$.
Substituting in our equation for $x$ in terms of $a$ we obtain, $$x^2+x+1=\left(\frac{a^2 -5a +7}{3}\right)\left(\frac{a^2+a+1}{3}\right).$$
We have then that $b^2+b+1 = a^2 -5a +7$. We can rewrite this equation to get that $$\left(2-a-b\right)\left(a-b-3\right)=0,$$
but that is impossible to satisfy since $a$ and $b$ are both odd primes.

We then finally have Case IV, where $x = \frac{a^2+a+1}{3}-(a+1)$ and $x=\frac{b^2+b+1}{3}-(b+1)$. We thus have
$$\frac{a^2+a+1}{3}-(a+1) = \frac{b^2+b+1}{3}-(b+1)$$
which implies that either $a=b$ (which we have already ruled out), or that $a+b=2$ which is impossible for primes $a$ and $b$.

\end{proof}

\section{3 Divides N}
We will in this section set $m=3^{\frac{f_3}{2}}ST$ where $$S= \prod_{p||m, p \neq 3} p$$ and $$T =  \prod_{p^2 | m, p \neq 3} p.$$ We will set $S=S_1S_2S_3S_4$ where a prime $p$ appears in $S_i$ for $1 \leq i \leq 3$ if $\sigma(p^2)$ is a product of $i$ primes; $S_4$ will contain all the primes of $S$ where $\sigma(p^2)$ has at least 4 prime factors. We will write  $s= \omega(S)$ and write $t= \omega(T)$. We define $s_1$, $s_2$ and $s_3$ similarly.
We will write $S_{i,j}$ to be the primes from $S_i$ which are $j$ (mod 3).  In a similar way to  use lower case letters to denote the number of primes in each term as before and in general will use a lower case letter to denote the number of distinct primes dividing an upper case letter. Foe example,, we set $s_{i,j}=\omega(S_{i,j})$ and will note that $s_{1,1}=0$. Thus, we do not need to concern ourselves with this split for $S_1$  since all primes in $S_1$ are 2 (mod 3) there is no need to split $S_1$ further.  We will abuse notation slightly and will treat capital letters as both products of distinct primes and as sets containing those distinct primes. Thus will also think of lower case letters as denoting the number of elements in the set formed by an upper case letter.

We have the special exponent is at least $1$:  \begin{equation} 1 \leq e \label{specialexists3case} \end{equation}

We have the following straightforward equations from breaking down the definitions of $s_1,s_2$, $s_3$ and $s_4$:

\begin{equation}
    s=s_1 +s_2 + s_3+s_4. \label{sbreakdown}
\end{equation}
Similarly, we have:  \begin{equation}
s_2=s_{2,1} + s_{2,2}, \label{s2breakdown}
\end{equation}  \begin{equation}
s_3=s_{3,1} + s_{3,2}, \label{s3breakdown}
\end{equation}
and
\begin{equation}
s_4=s_{4,1} + s_{4,2}. \label{s4breakdown}
\end{equation}
We  define $f_4$ as the number of prime divisors (counting multiplicity) in $N$ which are not the special prime and are raised to at least the fourth power. From simple counting we obtain:

\begin{equation} e+f_3+2s +f_4 \leq \Omega.  \label{bigomegalower2}
\end{equation}
Due to Lemma \ref{Other new lemma for repairing 3 divides N case} any element of $S_{3,1}$ must contribute at least one non-3 prime which is not from $S_1$. Motivated by this we will define $S_{3,1,T}$ to be the elements of $S_{3,1}$ which contribute two prime factors in $T$ or $e$, and define $S_{3,1,S}$, as those which contribute two prime factors in $S$. We will similarly define $S_{3,1,ST}$ as those which contribute one to $S$ and one which goes to $T$ or $e$. We will similarly define $S_{1,T}$ as the elements of $S_1$ which contribute to $T$ and define $S_{1,S}$ as the elements of $S_1$ which contribute to $S$. Define $S_{1,e}$ as the set of elements of $S_1$ which contribute the special prime.  We will correspondingly define $s_{3,1,T}$,  $s_{3,1,S}$, $s_{3,1,ST}$ $s_{1,T}$, $s_{1,S}$, and $s_{1,e}$. We of course have  have $s_{1,e} \leq 1$ but we will not need this here.  We have then:
\begin{equation}\label{s31breakdown}s_{3,1} \leq s_{3,1,T} + s_{3,1,S} + s_{3,1,ST}.
\end{equation}
We then define $S_{3,1,{\bar S}T}$ as the set of elements of $S_{3,1,ST}$ which have their $S$ contributing term not arising
from an $S_1$. Similarly define $S_{3,1,S{\bar T}}$ as the set of elements of  $S_{3,1,ST}$ which have their $T$ term not  arising from an $S_1$. We
define as usual their lower case variables for counting the number of elements in each set. From
Lemma \ref{Other new lemma for repairing 3 divides N case}, we have that $$S_{3,1,ST} = S_{3,1,{\bar S}T} \cup S_{3,1,S{\bar T}}. $$ Thus, we have:
\begin{equation}\label{s 31st breakdown} s_{3,1,ST} \leq s_{3,1,{\bar S}T} + s_{3,1,S{\bar T}}. \end{equation}
We also have
\begin{equation}\label{s1breakdown}s_1 \leq  s_{1,T} + s_{1,S} +s_{1,e}
\end{equation}

Here the $+1$ in the above inequality is due to the fact that we may have an element which contributes the special prime.

\begin{lemma} We have \begin{equation}
 s_1 + s_{2,2}  \leq t + s_{2,1}+ s_{3,1} +s_{4,1}+1. \label{s1s22upper}\end{equation}
\label{ts1s2}
\end{lemma}
\begin{proof}
The proof of this lemma is essentially the same as that of Lemma 4 from \cite{Zelinsky1}.
\end{proof}
Next we have \begin{equation}s_{2,1}+s_{3,1}+s_{4,1} \leq f_3, \label{3lower} \end{equation} since if $x \equiv 1$ (mod 3), then $x^2+x+1 \equiv 0$ (mod 3).

We also have by counting all the $1$ (mod $3$) primes which are contributed by primes in $S$ $$s_1+2s_{2,2}+3s_{3,2} + s_{2,1} + 2s_{3,1} + 4s_{3,2} + 3s_{3,1} \leq f_4+e + 2s_{2,1}+2s_{3,1} + 2s_{4,1}. $$ This simplifies to:
\begin{equation}s_1+2s_{2,2}+3s_{3,2} +4s_{4,2} + s_{4,1}   \leq f_4+e +s_{2,1}. \label{sbruteupper} \end{equation}

And we of course have \begin{equation}4t \leq f_4 \label{4tf4}.\end{equation}
We may split $S_{2,2}$ into four  sets, $S_{2,2,S}$,$S_{2,2,T}$,$S_{2,2,ST}$, and $S_{2,2,e}$. We define $S_{2,2,e}$ as the elements of $S_{2,2}$ which contribute to the special prime at least once.  We define $S_{2,2,S}$ as the elements of $S_{2,2}$ where both contribute to $S$, $S_{2,2,T}$ as the elements of $S_{2,2}$ where both contribute to $T$, and $S_{2,2,ST}$ as the elements of $S_{2,2}$ where one contributes to $S$ and one contributes to $T$.  We define their lower case variables as usual. We have:
\begin{equation} s_{2,2} \leq s_{2,2,S} + s_{2,2,T} + s_{2,2,ST} + s_{2,2,e}\label{s22 breakdown}. \end{equation}





We may split $S_{2,1}$ in a similar way into and $S_{2,1,S}$, $S_{2,1,T}$, and $S_{2,1,e}$.  We define the lower case variables as usual. We have
\begin{equation}\label{s21 breakdown} s_{2,1} \leq s_{2,1,S} + s_{2,1,T} + s_{2,1,e}.\end{equation}
We have \begin{equation} \label{General contribution to special prime}
s_{1,e} + s_{2,1,e} + 2s_{2,2,e} \leq e.  
\end{equation}
We have
\begin{equation}\label{s_1 breakdown} s_1 \leq s_{1,S}+ s_{1,T} +s_{1,e}.
\end{equation}

   

From Lemma \ref{Get no overlap between s1 and s21}, Lemma \ref{Only partial overlap with s1 and s22}, Lemma \ref{No s22 is an s1 and an s21} and Lemma \ref{No s22 is two s21s} we get that every element of $S_{2,2}$ must contribute at least one prime which does not arise from either an $S_1$ or an $S_{2,1}$.
Similarly, every element of $S_{2,2,ST}$ must have either a contribution to $T$ or $e$ which does arise from an $S_{2,1}$ or $S_1$ element  or must have a contribution to  $S$  which does not arise from an $S_{2,1}$ or an $S_1$ element. We will set $S_{2,2,S*T}$ as those which contribute an $S$ element of this form, and $S_{2,2,ST*}$ as those which contribute a $T$ term. We define the lower-case counting variables as usual. We have then:
\begin{equation} S_{2,2,ST} \leq S_{2,2,S*T} + S_{2,2,ST*}. \label{star s and star t breakdown} \end{equation}
We also have
\begin{equation} 2s_{1,S} + 2s_{2,1,S} + s_{2,2,S}  +S_{2,2,S*T} + s_{3,1,S} + s_{3,1,\bar{S}T} \leq 2s_{2,1} + 2s_{3,1} + 2s_{4,1}. \label{star s inequalty}
\end{equation}
and
\begin{equation} 4s_{1,T} + 4s_{2,1,T} + s_{2,2,T}  + s_{2,2,ST*} + s_{3,1,T} +  s_{3,1,S\bar{T}}  \leq f_4 +e \label{star t inequality}
\end{equation}
We have from counting the primes from $S$ which are contributed by $S$
\begin{equation}\label{s to s} s_{1,S} + s_{2,1,S} + 2s_{2,2,S} + 2s_{3,1,S} + s_{3,1,ST} \leq 2s_{2,1} + 2s_{3,1} + 2s_{4,1}.
\end{equation}

We have that
\begin{equation}\label{littleomegabreakdownwith3} s+t+2=\omega.
\end{equation}

Note that the $+2$ in Equation \ref{littleomegabreakdownwith3} arises from 3 and the special prime.

To prove the result we add up our inequalities as follows:\\
$ \frac{9}{25}\times ({\bf{\ref{specialexists3case}}})  +\frac{16}{25}\times ({\bf{\ref{sbreakdown}}}) +\frac{16}{25}\times ({\bf{\ref{s2breakdown}}}) + \frac{16}{25}\times ({\bf{\ref{s3breakdown}}}) +\frac{16}{25}\times ({\bf{\ref{s4breakdown}}}) + 1\times({\bf{\ref{bigomegalower2}}}) + \frac{2}{25}\times ({\bf{\ref{s31breakdown}}}) +\frac{1}{25}\times ({\bf{\ref{s 31st breakdown}}}) + \frac{8}{25}\times ({\bf{\ref{s1breakdown}}}) +\frac{2}{25} \times  ({\bf{\ref{s1s22upper}}}) + 1  \times ({\bf{\ref{3lower}}})+ \frac{6}{25}\times ({\bf{\ref{sbruteupper}}}) + \frac{17}{25} \times ({\bf{\ref{4tf4}}}) + \frac{2}{25} \times ({\bf{\ref{s22 breakdown}}})  + \frac{8}{25}\times ({\bf{\ref{s21 breakdown}}}) + \frac{8}{25} \times  ({\bf{\ref{General contribution to special prime}}}) +\frac{1}{25}\times ({\bf{\ref{star s and star t breakdown}}}) +  \frac{7}{50} \times ({\bf{\ref{star s inequalty}}}) + \frac{2}{25}\times ({\bf{\ref{star t inequality}}})  +\frac{1}{25}\times ({\bf{\ref{s to s}}})   +\frac{66}{25}\times ({\bf{\ref{littleomegabreakdownwith3}}})   $


\section{3 does not divide N}
For simplicity, we will prove the slightly weaker bound that
\begin{equation}\label{Weak form of 3 does not divide N bound} \frac{302}{113}\omega -\frac{641}{113} \leq \Omega\end{equation} and then discuss the changes needed to improve the constant term.
We set $m=5^{\frac{f_5}{2}}11^{\frac{f_{11}}{2}}S^2T^4U'$. Here we have
$$S= \prod_{p,p||m,p \not\in\{5,11\}} p, T    = \prod_{p,p^2||m,p \not\in\{5,11\}} p. $$ We have $U=\frac{m}{5^{\frac{f_5}{2}}11^{\frac{f_{11}}{2}}S^2T^4}$ and $U=rad(U')$. That is, $$U =  \prod_{p,p^3|m,p \not\in\{5,11\}} p.$$
In other words, $S$ contains the prime divisors other than 5 and 11 which are raised to exactly the second power in the factorization of $N$. Similarly, $T$  contains the prime divisors other than 5, and 11 which are raised to exactly the fourth power in the factorization of $N$. Finally,  $U$ contains the prime divisors other than 5, and 11 and the special prime which are raised to at least the sixth in the factorization of $N$.
We set $s=\omega(S)$, $t=\omega(T)$ and $u=\omega(U)$.

We have then
\begin{equation}\label{little omega break down in (3,N)=1 case} \omega \leq s+ t +u  - 3. \end{equation}

We similarly define $f_6$ as the set of primes (counting multiplicity) who appear to at least the 6th power. We have then:
\begin{equation}\label{u bound 1} 6u \leq f_6 \end{equation}
and
\begin{equation} \label{bigomegabreakdown no 3s} e + f_5  +2s + 4t + f_6 +f_{11}\leq \Omega .  \end{equation}
We do not have an equality in Equation \ref{u bound 1} because there may be an element of $u$ raised to a power higher than 6.
We will define $S_T$ to be the elements of $S$ which arise from $T$, that is $$S_T= \prod_{p|S,p|\sigma(T^4)}p$$.
Similarly, we will define $T_S$ to be the elements of $T$ which arise from $S$.  $$T_S=\prod_{p|T,p|\sigma(S^2)}p.$$
We note that any prime in $S_T$ must be $1$ (mod 5) or must be 5 itself. We will define $S_M$ as the primes in $S$ which are $1$ (mod 5), and set $s_{M5}= \omega(S_{M5})$.
We will set $S=S_1S_2S_3S_4S_5$ similarly to how we did in the case of $3|N$, with $p|S_5$ if $\sigma(p^2)$ has 5 or more not necessarily distinct prime factors, and similarly define $s_1, s_2, s_3, s_4$ and $s_5$. We then define $S_{M1}, S_{M2} \cdots S_{M5}$ as the intersections of the corresponding $S_i$ and $S_M$, and then define $s_{M1}, s_{M2} \cdots s_{M5}$ accordingly.
We have
\begin{equation}\label{s breakdown for (3,N)=1 case} s\leq s_1 +s_2 + s_3 + s_4 + s_5\end{equation}
We have
\begin{equation} s_M \leq s_{M1} + s_{M2} + s_{M3} + s_{M4} +s_{M5} .\label{SM breakdown}
\end{equation}
We have that each of the $s_{i}$ is at least $s_{Mi}$ and thus we have:
\begin{equation} \label{s2 is at least sm2} s_{M2} \leq s_2
\end{equation}
\begin{equation} \label{s3 is at least sm3} s_{M3} \leq s_3
\end{equation}
\begin{equation} \label{s4 is at least sm4} s_{M1} \leq s_4
\end{equation}
\begin{equation} \label{s5 is at least sm5} s_{M5} \leq s_5
\end{equation}

Define $T_S$ to be the set of elements of $T$ which are contributed by $S$, and define $t_S$ as the lower case variable as usual. Note that every element of $T_S$ is 1 (mod 3). We note that
\begin{equation}\label{t_s lower bound for t}t_S \leq t.
\end{equation}
We note that any prime factor contributed by $S$ cannot itself be in $S$. To see why, note that any prime $p$ contributed by $S$ is $1$ (mod 3), and in that case $3|\sigma(p^2)$, but we are assuming in this section that $3 \not |N$.
We have from Lemma \ref{No triple threats with x = 1 mod 5 and any of a b or c 1 mod 5} that every element of $S_{3M}$ must either have all contributions be terms  which do not arise from an $S_{1M}$ or must contribute a term which does not arise from an $S_1$ at all. We will set $S_{3MA}$ as the set of elements of $S_{3N}$ which contribute all terms not arising from $S_{1M}$ and $S_{3MB}$ as the set of elements of $S_{3M}$ which do not contribute at least one term which does not arise from $S_1$. We define the lower case counting variables as usual. We have then
\begin{equation}\label{s3m breakdown} s_{3M} \leq s_{3MA} + s_{3MB}.\end{equation}
We have then
\begin{equation}\label{s3ma bound1} 4s_{1,m} + 3s_{3MA} + s_{3MB} +s_2-3 \leq 4t_S + u_S + e_S.
\end{equation}
Similarly we have
\begin{equation}\label{s3mb bound1} 4s_1 + s_{3MB} +s_2-3 \leq 4t_S + u_S + s_e.
\end{equation}
We also have by the same logic as  Lemma 4 from \cite{Zelinsky1}.
\begin{equation} \label{s1 + s2 for 3 not dividie N case} s_1 + s_2 \leq t_S + u_S + 1.
\end{equation}
We also have from counting all the various $s_i$ contributions:
\begin{equation}\label{brute s_i contribution for 3 not divide N} s_1 + 2s_2 + 3s_3 +4s_4 + 5s_5 \leq 4t_S + u_S + e_S .
\end{equation}
We now turn to the equations which allow us to bound the number of primes in $S_M$. To do this we need lower bounds on the contribution to $S$ from elements in $T$.
We define $T_{S1}$, $T_{S2}$, $T_{S3}$, $T_{S4}$ and $T_{S5}$ as follows: For $1 \leq i \leq 4$ we define $T_{Si}$ to be the elements of $T_S$ which contribute exactly $i$ primes, and we define $T_{S5}$ to be those elements of $T_S$ which contribute at least 5 primes.   We define $T_M$ to be the set of elements of $T$ which are 1 mod 5.
\begin{equation}\label{T breakdown} t_S \leq t_{S1} + t_{S2} + t_{S3} + t_{S4} + t_{S5} .\end{equation}

We have
\begin{equation}\label{ts brute}  t_{S1} + 2t_{S2} + 3t_{S3} +4t_{S4} + 5t_{S5} \leq 2s_M + 4t_M + u_T + e_T + f_5 + f_{11}\end{equation}
where $u_t$ and $e_t$ are defined analogously to $u_s$ and $e_s$.

We note that every element in $T_S$ is 1 (mod 3), and that if $x \equiv 1 $ (mod 3), we have that $x^4+x^3+x^2+x+1 \equiv 2$ (mod 3). But every element of $S$ must be $1$ (mod 3). So  if $p \in T_S$ then $\sigma(p^4) \equiv 2$ (mod 3).  Thus, any element of $T_{i,S}$ when $i$ is odd must contribute at least one prime which is $1$ (mod 3 (and hence contribute a prime not in $S$), since a product of an even number of numbers of the form $2$ (mod 3) will be $1$ (mod 3). Thus we have:
\begin{equation} \label{t_even contribution} t_{S2} + t_{S4} \leq 4t_M + u_t + e_t.
\end{equation}
The next set of inequalities seeks to deal with the problem that we may have very large $T_{S1}$. From Lemma \ref{T1 to S2 contributes an 11} it follows that  no element of $T_{S1}$ can give rise to an element of $S_1$. For $i$ with $2 \leq i \leq 5$ Define then $S_{i*}$ as the elements of $S_i$ arising from $T_{1S}$. We define the lower case counting variables as usual. We have then
\begin{equation}\label{ts1 gives rise to miraculous factorizations} t_{S1} \leq s_{2*} + s_{3*} + s_{4*} + s_{5*} + t_M + u_t + e_t.
\end{equation}
We have 
\begin{equation}\label{s4* < s4} s_{4*} \leq s_4, \end{equation}
and \begin{equation}\label{s5* < s5} s_{5*} \leq s_5. \end{equation}
We have from Lemma \ref{Big piece of T1 to S is an S1 cannot happen} and  Lemma \ref{Big piece of T1 to S is from a small T1 to S cannot happen} that
\begin{equation}\label{s_1 + s_3*} s_1 + s_{3*} - 1 \leq t_S + u. \end{equation}
We have from counting our contribution to the special prime that
\begin{equation}\label{es + et <= e}  e_S + e_T \leq e. \end{equation}
We also have
\begin{equation} \label{f_6 lower bound}  u_S + u_T \leq f_6.  \end{equation}
We note that since every element in $T_M$ is 1 mod 5, we have
\begin{equation}\label{t_m to f_5} t_m \leq f_5\end{equation}
We also have by Lemma \ref{T1 to S2 contributes an 11} that any contribution from $S_{2*}$ which contributes both terms to $T$ must also contribute to $f_{11}$. We set $S_2*T$ to be those elements of $S_2*$ which contribute both terms to $T$, and we set $S_{2*UE}$ to be those which contribute at least one term to either $u$ or $e$. We have then:
\begin{equation}\label{s_2* breakdown} s_{2*} \leq s_{2*T} + s_{2*UE},
\end{equation}
\begin{equation}\label{s_2* to T forces 11} s_{2*T} \leq f_{11}, \end{equation}
\begin{equation}\label{s_2*UE contribution} s_{2*UE} \leq u-1.  \end{equation}

We have as before that the special prime must be raised to at least the first power and thus:
\begin{equation}\label{special prime exists in 3 not divide N case} 1 \leq e. \end{equation}
Finally we have that \begin{equation}\label{omega breakdown for general 3 does not divinde N} \omega \leq s+t+u +3. \end{equation}
The 3 comes from the special prime, the possibility of division by 5 and the possibility of division by 11.

To prove Inequality \ref{Weak form of 3 does not divide N bound} we then add our inequalities as follows:
$ \frac{302}{113} \times ({\bf{\ref{little omega break down in (3,N)=1 case}}}) +  \frac{53}{113} \times ({\bf{\ref{u bound 1}}}) + 1\times ({\bf{\ref{bigomegabreakdown no 3s}}}) +\frac{76}{113} \times ({\bf{\ref{s breakdown for (3,N)=1 case}}})    +  \frac{8}{113} \times ({\bf{\ref{SM breakdown}}}) + \frac{8}{113} \times ({\bf{\ref{s2 is at least sm2}}})  +\frac{2}{113} \times ({\bf{\ref{s3 is at least sm3}}}) + \frac{8}{113} \times ({\bf{\ref{s4 is at least sm4}}}) + \frac{8}{113} \times ({\bf{\ref{s5 is at least sm5}}}) + \frac{150}{113} \times ({\bf{\ref{t_s lower bound for t}}}) + \frac{6}{113} \times ({\bf{\ref{s3m breakdown}}})  + \frac{2}{113} \times ({\bf{\ref{s3ma bound1}}})  + \frac{4}{113} \times ({\bf{\ref{s3mb bound1}}})  + \frac{26}{113} \times ({\bf{\ref{s1 + s2 for 3 not dividie N case}}}) + \frac{26}{113} \times ({\bf{\ref{brute s_i contribution for 3 not divide N}}}) + \frac{12}{113} \times  ({\bf{\ref{T breakdown}}}) + \frac{4}{113}\times ({\bf{\ref{ts brute}}}) + \frac{4}{113} \times ({\bf{\ref{t_even contribution}}})  \frac{8}{113} \times({\bf{\ref{ts1 gives rise to miraculous factorizations}}}) +  \frac{8}{113}\times({\bf{\ref{s4* < s4}}}) + \frac{8}{113}\times({\bf{\ref{s5* < s5}}}) + \frac{8}{113}   \times ({\bf{\ref{s_1 + s_3*}}}) + \frac{32}{113} +  \times ({\bf{\ref{es + et <= e}}})   \frac{58}{113} \times ({\bf{\ref{f_6 lower bound}}}) +  \frac{40}{113} \times ({\bf{\ref{t_m to f_5}}}) + + \frac{8}{113} \times ({\bf{\ref{s_2* breakdown}}})  +  \frac{8}{113} \times ({\bf{\ref{s_2* to T forces 11}}})  +\frac{8}{113} \times ({\bf{\ref{s_2*UE contribution}}})  + \frac{81}{113} \times ({\bf{\ref{special prime exists in 3 not divide N case}}}) $
It follows from the results in \cite{Fletcher Nielsen and Ochem} that if one has any of $5^2||N$, $5^4||N$, $11^2||N$, or $11^4||N$ then one must have at least one prime which is not the special prime raised to at least the 6th power, and hence have $u \geq 1$.
We thus may adjust the above equations slightly: If $(N,55)=1$, we have instead of \ref{little omega break down in (3,N)=1 case}, we have $\omega =  s+t+u -1$ and $f_5=f_{11}=0$. In this case we get a bound of \begin{equation}\label{final bound form for (3,N)=1} \frac{302}{113}\omega - \frac{286}{113} \leq  \Omega.\end{equation}
Alternatively, one must have $5^6|N$ or $11^6|N$ if either $5|N$ or $11|N$.
We can without too much work check that all of these force a bound at least as tight as \ref{final bound form for (3,N)=1}. Thus we always have that bound.

\section{Improved Norton type results}
Norton \cite{Norton} proved two types of results. First, he proved lower bounds for $\omega(N)$ in terms of the smallest prime factor of $N$. Second, he proved lower bounds for $N$ in terms of its smallest prime factor. In this section, we will slightly improve Norton's first type of result and show how we can combine that with the Ochem-Rao type results to substantially improve the second type of result.
Set $P_n$ to be the $n$th prime number. For $n>1$ Norton defined $a(n)$ as the integer such that
\begin{equation} \prod_{r=n}^{n+a(n)-2} \frac{P_r}{P_r-1} < 2<  \prod_{r=n}^{n+a(n)-1} \frac{P_r}{P_r-1}. \label{Norton-Grun definition}
\end{equation}
It is easy to see that if $N$ is an odd perfect number with smallest prime divisor $P_n$, then $N$ must have at least $a(n)$ distinct prime divisors. In fact, although Norton does not state this explicitly, this statement also applies to any odd abundant number $N$. Thus, study of $a(n)$ is a natural object even if one is not strongly interested in odd perfect numbers.
Norton proved
\begin{theorem}(Norton)\cite{Norton}\label{Summary of Norton's nonconstructive results} Let $N$ be an odd perfect number with smallest prime divisor $P_n$, largest prime divisor $P_s$,  and let $b$ be a constant less than $\frac{4}{7}$. Then we have:
\begin{enumerate}
    \item \begin{equation}\label{Norton 4-1}a(n)= Li(P_n^2) + O\left(n^2e^{-\log^b n}\right).\end{equation}
    \item \begin{equation} \label{Norton 5} a(n) = \frac{1}{2}n^2 \log n + \frac{1}{2} n^2 \log \log n -\frac{3}{4}n^2 + \frac{n^2 \log \log n }{2 \log n} + O\left(\frac{n^2}{\log n}\right)\end{equation}
    \item  \begin{equation}\label{Norton 4-2} P_s \geq  P_{n+a(n)-1} = P_n^2 + O\left(n^2e^{-\log^b n}\right).\end{equation}
    \item \begin{equation}\label{Norton 6} P_s \geq P_{n+a(n)-1} = n^2 \log^2 n + 2n^2 \log n \log \log n - 2n^2 \log n + n^2 (\log \log )^2 +O(n^2) \end{equation}
    \item \begin{equation}\label{Norton 4-3}\log N > 2P_n^2 + O\left(n^2e^{-\log^b n}\right).\end{equation}
\end{enumerate}
\end{theorem}
Note that only Equation \ref{Norton 4-3} is using non-trivial material about odd perfect numbers. The bounds for $P_s$ apply to any odd abundant number, and the bounds for $a(n)$ do not depend on odd perfect numbers at all. These bounds of Norton are not by themselves constructive; Norton proved slightly weaker constructive bounds.\footnote{Prior to Norton a similar non-constructive but more general bound was proven which gives a lower bound for the $\omega(n)$ in terms of $\alpha$ and $p$ where $n$ satisfies $\frac{\sigma(n)}{n} \geq \alpha$  and  $n$ has smallest prime factor $p$.\cite{Sale}}
\begin{theorem}(Norton)\cite{Norton}\label{Summary of Norton's constructive results} Let $N$ be an odd perfect number with smallest prime divisor $P_n$, and set $P_s= P_{a(n)-n-1}$. Then we have
\begin{enumerate}
    \item \begin{equation}\label{Norton 1-1}a(n) >n^2 -2n - \frac{n+1}{\log n} - \frac{5}{4} - \frac{1}{2n} - \frac{1}{4n \log n}. \end{equation}
    \item As long as $n \geq 9$, \begin{equation}\label{Norton's improved version of Theorem 3} \log N > 2P_s\left(1- \frac{1}{2\log P_s}\right) - 2P_n\left(1 + \frac{1}{2\log P_n} \right) + 6 \log P_n + 2 \log P_{n+1} - \log P_s .\end{equation}
\end{enumerate}
\end{theorem}
These bounds are explicit, with the cost of being weaker in form than the bounds in Theorem \ref{Summary of Norton's nonconstructive results}. To prove Equation \ref{Norton 1-1} Norton did have to use results about odd perfect numbers. In fact, Norton's method uses some early results ruling out specific forms of odd perfect numbers, although the forms ruled out only allow improvement in the lower order terms of his inequality. Subsequent results in this section can be thought of as using similar ideas, but because we have stronger results on what an odd perfect number can look like, we can actually improve the constant in front of the lead terms.
It is also worth noting that for large values of $n$, Norton's explicit bound for $a(n)$  give a better result than the often cited bound of Grun \cite{Grun}, which shows that if $N$ is an odd perfect number with smallest prime divisor $p$, then \begin{equation}\label{Grun}(3/2)p -2 \leq \omega(N) .\end{equation}
Norton's result focus on $a(n)$, but  for some  purposes it is more natural to look at the function $b(p)$, defined for odd primes $p$, where $b(p)=a(n)$ where $p=P_n$. We will examine the behavior of both functions in this section.
While Norton's results in work for abundant or perfect numbers, we will also be interested in how they can be strengthened for odd perfect numbers. In this context, we will define $b_o(p)$ to be the minimum of the number of distinct prime divisors of any odd perfect number with smallest prime divisor $p$. We will set $b_o(p)= \infty$ when there are no odd perfect numbers with smallest prime divisor $p$. We will define $a_o(n)$ in analogous fashion. Trivially one has $b_o(p) \geq b(p)$ and $a_o(n) \geq a(n)$. One would like to be able to prove that $b_o(p) > b(p)$ but this seems very difficult.

In this section, we will do three things: First, we will prove strengthened versions of Norton's constructive results bounding $a(n)$ from below and a similar one for $b(p)$ although they will still fall slightly short of the non-constructive results. Second, we will construct a general framework to use Ochem and Rao type results to get explicit inequalities similar to \ref{Norton's improved version of Theorem 3} which in general are better than Norton (both his explicit form and his non-constructive form in Equation \ref{Norton 4-3}). Third, we will use this framework and our earlier results to construct a strong lower bound for the size of an odd perfect number in terms of its smallest prime factor.
We will write $S(x)=\prod_{p \leq x}\frac{p}{p-1}.$
We will write $\vartheta(x)$ to be Chebyshev's first function, that is $$\vartheta(x)=\sum_{p\le x} \log p. $$
We need the following lemma:
\begin{lemma} For any prime $p>2$, we have $b(p) \geq p$.
\label{Servais lemma}
\end{lemma}
\begin{proof} Assume that $b(p)=m$. Then since $\frac{x}{x-1}$ is a decreasing function for positive $x$ and the $i$th prime after $p$ is at least $p+i$, we must have
$$2 <  \left(\frac{p}{p-1}\right)\left(\frac{p+1}{p} \right)\left(\frac{p+2}{p+1}\right) \cdots \left(\frac{p+b(p)-1}{p+b(p)-2}\right) = \frac{p+b(p)-1}{p-1}.  $$
Thus $$2p-2 < p+b(p) -1,$$
and so $b(p) > p-1$, and so $b(p) \geq p$.
\end{proof}
Note that variants of the above lemma are very old. Often this result is normally stated simply that for an odd perfect number needing to at least as many distinct prime factors than their smallest prime divisor.  The lemma when stated applying just to perfect numbers seems to back to Servais\cite{Servais} and has been proved repeatedly such as in \cite{Pambuccian}. In that context, it is worth noting that there is also a slightly stronger version in the literature which again has been proven a few times. In this case, the oldest version appears to be due to Grun \cite{Grun}. As previously discussed, Grun proved that if $N$ is an odd perfect number, with least prime divisor $p$, then $\frac{2}{3}\omega(N) +2 \geq p $. Again, the proof can, with no substantial effort, be generalized to a statement about $b(p)$. In Grun's proof, the key  observation that odd primes must differ by at least  2, and therefore  can instead use the inequality  $$2 <  \left(\frac{p}{p-1}\right)\left(\frac{p+2}{p+1} \right)\left(\frac{p+4}{p+3}\right) \cdots \left(\frac{p+2b(p)-2}{p+b(p)-3}\right), $$ and then estimate this quantity. As with the lemma of Servais, Grun's lemma applies to any odd perfect or odd abundant number although it is normally phrased simply for odd perfect numbers.

We will need a few explicit estimates of certain functions of primes.
We have \cite{Dusart 1999} that \begin{equation}\label{Dusart 1999 pi bound} \frac{x}{\log x}\left(1+ \frac{0.992}{\log x}\right) \leq \pi(x) \leq \frac{x}{\log x}\left(1 + \frac{1.2762}{\log x} \right),
\end{equation}
with the lower bound valid if $x \geq 599$ and the upper bound valid for all $x > 1$.
We also have for $n \geq 2$, \begin{equation}\label{Modified Dusart Pn bound}  P_n \geq n\left(\log n + \log \log n -1 + \frac{32}{31 (\log n)^2} \right).  \end{equation}
The above bound for $P_n$ follows from the following bound in \cite{Dusart 1999} which differs only in the last term:
\begin{equation}\label{Modified Dusart Pn bound}  P_n \geq n\left(\log n + \log \log n -1 + \frac{\log \log x - 9/4}{\log x}. \right).  \end{equation}
To obtain \ref{Modified Dusart Pn bound}, we note that for if $n \geq 35312$, we have that $$\frac{\log \log n - 9/4}{\log n} \geq \frac{32}{31 (\log n)^2}, $$ and we may then verify the inequality for all $n$ with $2 \leq n \leq  35312$.  We will write $$\underline{P}_n= n\left(\log n + \log \log n -1 + \frac{32}{31 (\log n)^2}\right).$$
\begin{lemma} If $A$ and $B$ are real numbers, and $A \geq B \geq 6$
satisfying $$A + \frac{1}{2A} \geq 2B - \frac{1}{B},$$ then $$A \geq 2B - \frac{5}{4B} - \frac{1}{37B^2}.$$
\label{Lemma to use Mertens v1.1}
\end{lemma}
\begin{proof} Note that $A + \frac{1}{2A}$ is a positive increasing function in $A$, to prove this one simply needs verify that if $B \geq 6$ then $$2B - \frac{5}{4B} - \frac{1}{37B^2} + \frac{1}{2(2B - \frac{5}{4B} - \frac{1}{37B^2})} \leq  2B - \frac{1}{B}.$$
\end{proof}
For $x >1$ we have \cite{Rosser}
\begin{equation} \label{Rosser S(x) bound}  e^\gamma (\log x) \left( 1 - \frac{1}{2 \log^2 x} \right) < S(x) <  e^\gamma (\log x) \left( 1 + \frac{1}{ \log^2 x} \right) \end{equation} where $\gamma$ is Euler's constant.
We will write $$Q(x,y)= \prod_{y < p \leq x} \frac{p}{p-1}=\frac{S(x)}{S(y)}$$ and will assume that $x>y$. We have from Equation \ref{Rosser S(x) bound} that
\begin{equation}\label{Q(x) basic bound} \frac{\log x}{\log y}\left(\frac{1-\frac{1}{2 \log^2 x}}{1+ \frac{1}{\log^2 y}} \right) < Q(x,y) < \frac{\log x}{\log y}\left(\frac{1+\frac{1}{\log^2 x}}{1- \frac{1}{2\log^2 y}} \right).\end{equation}
\begin{theorem} For all $n>1$ we have $$a(n) \geq \frac{n^2}{2}\left(\log n - \frac{3}{2}\log \log n +\frac{1}{20} + \frac{\log \log n}{\log n} \right) - n +1 . $$ \label{Improved a(n) bound}
\end{theorem}
\begin{proof} We may verify from direct computation that the above inequality is valid for $n \leq 117$, and so we may assume that $n \geq 118$  (equivalently that $p \geq 647$).
Our plan to estimate $a(n)$ is to set $y=P_n$ in Equation \ref{Q(x) basic bound}, find a lower bound on $x$ such that $Q(x,y)>2$, and then estimate $\pi(x)$ since we will have $a(n) = \pi(x) - n +1$.  We have from Equation \ref{Q(x) basic bound} that $$2 < \frac{\log x}{\log y}\left(\frac{1+\frac{1}{\log^2 x}}{1- \frac{1}{2\log^2 y}} \right)  $$ which implies that
$$\log x + \frac{1}{2\log x} \geq 2\log y - \frac{1}{\log y}.$$
We have then from Lemma \ref{Lemma to use Mertens v1.1} that \begin{equation}\label{log x lower bound in terms of log y}\log x \geq 2 \log y - \frac{5}{4 \log y} - \frac{1}{37 (\log y)^2}. \end{equation} (In the above use of the Lemma we are setting  $A= \log x$ and $B=\log y$.)
Note that getting a lower bound $a(n)$ is exactly the same as lower bounding the minimum $x$, such that $Q(x,p_n)>2$, and then getting a lower bound on $\pi(x) - \pi(y)=\pi(x)-n+1$.
It is not hard to see that Equation \ref{log x lower bound in terms of log y} yields that as long as $y \geq 541$ that \begin{equation}\label{x lower bound} x \geq y^2\left(1 - \frac{5}{4\log y} + \frac{1}{2 (\log y)^2}\right). \end{equation}

We now need to estimate $a(n)$ by estimating $\pi(x)-n+1$.  We have then from Equation \ref{x lower bound} and Equation \ref{Modified Dusart Pn bound}, along with the fact that the function $j(s)= 1 - \frac{5}{4s} + \frac{1}{2s^2}$ is increasing for $t \geq 1/10$ that
\begin{equation}a(n) \geq \pi\left((\underline{P}_n)^2\left(1 - \frac{5}{4\log \underline{P}_n} + \frac{1}{2(\log \underline{P}_n)^2}\right)\right) - n -1  \end{equation}
We have the trivial estimate that $\underline{P}_n \geq n$, and so we have again using that $f(s)$ is increasing in $s$, and substituting in the definition of $\underline{P}_n$, and set $t= \log n$ that

\begin{equation}a(n) \geq \pi\left(n^2\left(t + \log t -1 + \frac{32}{31 t^2}\right)^2 \left(1 - \frac{5}{4 t} + \frac{1}{2t^2}\right)\right) - n -1  \label{a(n) estimate with sub in t}. \end{equation}
When $t>4.77$, one has that $$ \left(t + \log t -1 + \frac{32}{31 t^2}\right)^2 \left(1 - \frac{5}{4 t} + \frac{1}{4t^2}\right) \geq t^2 - \frac{\log t}{2}.$$
So for $n$ in the range under discussion we have
\begin{equation}\label{second a(n) estimate with sub in t}a(n) \geq \pi\left(n^2(t^2 - \frac{\log t}{2}) \right) -n +1 = \pi\left(n^2 ((\log n)^2 - \frac{\log \log n}{2}) \right) -n +1.
\end{equation}
We wish to apply Equation \ref{Dusart 1999 pi bound} to \ref{second a(n) estimate with sub in t}. To do so, we need a lower bound on $$\frac{1}{\log \left(n^2 ((\log n)^2 - \frac{\log \log n}{2}) \right) }.$$
It is not hard to check that as long as $n>e^e$ one has that
\begin{equation}
    \frac{1}{\log \left(n^2 ((\log n)^2 - \frac{\log \log n}{2}) \right) } \geq \frac{1}{2 \log n}\left(1- \frac{\log \log n}{\log n}\right). \label{I(t) lower bound}
\end{equation}
We can now apply Equation \ref{I(t) lower bound} to Equation \ref{second a(n) estimate with sub in t} and \ref{Dusart 1999 pi bound} to get that
\begin{equation}a(n) \geq \frac{n^2\left(t^2 - \frac{t \log t}{2}\right)\left(1-\frac{\log t}{t}\right)}{2t}\left(1+ .992\left(\frac{1}{2t}- \frac{\log t}{2t^2}\right)\right) - n -1. \label{Second to last version of a(n) bound}
\end{equation}
We have again in the above for convenience written $\log n$ as $t$. A little work then shows that for $n \geq 43$, we have that the right-hand side of Equation \ref{Second to last version of a(n) bound} is at least $$\frac{n^2}{2}\left(\log n - \frac{3}{2}\log \log n +\frac{1}{20} + \frac{\log \log n}{\log n} \right) - n +1$$
which proves the theorem.
\end{proof}
Note that the bound given by Theorem \ref{Improved a(n) bound} is tighter than the bound from Equation \ref{Norton 1-1}.
We similarly have an interest in estimating $b(p)$. We have then:
\begin{theorem} For all primes $p>2$, we have \label{b(p) bound theorem}
\begin{equation}
b(p) \geq \frac{p^2}{2\log p}\left(1 - \frac{0.754}{\log p}  - \frac{0.745}{(\log p)^2} - \frac{0.247}{(\log p)^3} + \frac{0.631813}{(\log p)^4}\right) - \pi(x) +1. \label{b(p) bound with Pi term}
\end{equation}
and
\begin{equation} b(p) \geq \frac{p^2}{2\log p}\left(1 - \frac{0.754}{\log p}  - \frac{0.745}{(\log p)^2} - \frac{0.247}{(\log p)^3} + \frac{0.631813}{(\log p)^4}\right) - \frac{p}{\log p}(1 + \frac{1.2726}{\log p}) +1.  \label{b(p) bound}
\end{equation}
\end{theorem}

\begin{proof} We will prove only the second of the two inequalities (the proof for the first statement is nearly identical). The first few steps in this proof are essentially identical to those in the proof of Theorem \ref{Improved a(n) bound}. We again assume that $p \geq 647$, and proceed until we reach Equation \ref{x lower bound}. And as before we estimate $$\pi(x) - n +1.$$ We need to lower bound the left-hand side of \begin{equation}b(p) \geq  \pi
\left(p^2\left(1 - \frac{5}{4\log p} + \frac{1}{2 (\log p)^2}\right)\right) - \pi(p) +1. \label{pi(p) to apply Dusart to}\end{equation} We need to apply \ref{Dusart 1999 pi bound}. We note that although the lower bound on \ref{Dusart 1999 pi bound} requires that the argument of $x$ be at least 599, we have that in this case since $p \geq 647$. We also need a lower bound estimate for
$$\frac{1}{\log(p^2(1- \frac{5}{4 \log p} + \frac{1}{2 (\log p)^2})}. $$
It is not hard to verify that when $p$ in our range we have
\begin{equation} \frac{1}{\log(p^2(1- \frac{5}{4 \log p} + \frac{1}{2 (\log p)^2})} \geq \frac{1}{2 \log p}\left(1- \frac{5}{8 (\log p)^2} - \frac{21}{32 (\log p)^3} \right). \label{Lower bound for 1 over log needed in Dusart}
\end{equation}
We will set $t=\log p$, and then use Equation \ref{Lower bound for 1 over log needed in Dusart} to apply Equation \ref{Dusart 1999 pi bound} to  \ref{pi(p) to apply Dusart to} to get that:
\begin{equation} b(p) \geq \frac{p^2}{2t}\left(1- \frac{5}{4t} + \frac{1}{2t^2}\right)\left(1 - \frac{5}{8t^2} - \frac{21}{32t^3}\right)\left(1 + \frac{0.992}{2t}\left(1 - \frac{5}{8t^2} - \frac{21}{32t^3}\right)\right)  - \frac{p}{t}(1 + \frac{1.2726}{t}) +1. \label{antepenultimate b(p) bound}
\end{equation}
We need then to estimate $$I_2(t)=\left(1- \frac{5}{4t} + \frac{1}{2t^2}\right)\left(1 - \frac{5}{8t^2} - \frac{21}{32t^3}\right)\left(1 + 0.992\left(1 - \frac{5}{8t^2} - \frac{21}{32t^3}\right)\right). $$
We have $$I_2 =  1 - \frac{0.754}{t} - \frac{0.745}{t^2} - \frac{0.247}{t^3} + \frac{0.631813}{t^4} + E(t)$$
Where $$E(t) = \frac{0.369375}{t^5} -  \frac{0.160813}{t^6} - \frac{0.198109}{t^7} - \frac{0.0635742}{t^8} +  \frac{0.106805}{t^9}.$$
We note that $E(t)$ is positive when $t>1$ which is satisfied in our range.
 Thus we conclude that
$$I_2(t) \geq  1 - \frac{0.754}{t} - \frac{0.745}{t^2} - \frac{0.247}{t^3} + \frac{0.631813}{t^4}  $$
which proves the theorem.
\end{proof}
Note that although Theorem \ref{b(p) bound theorem} and Theorem \ref{Improved a(n) bound} both give the asymptotically correct values, in practice Theorem \ref{b(p) bound theorem}  is stronger. This is due to Theorem \ref{Improved a(n) bound} requiring that we use a lowest bound estimate for $P_n$ in terms of $n$.

While we cannot directly show that $b_o(p) > b(p)$, we can get partial results of this form.  In particular, $b(p)=3$, but $b_o(3) \geq 10$ \cite{Nielsen}. Similarly, $b(5)=7$, and $b_o(5) \geq 12$ \cite{Nielsen 3}. We will also prove a similar result for other small values of $b(p)$ using the fact  that the largest prime divisor of an odd perfect number must be at least $10^8$ \cite{Goto and Ohno}.
\begin{proposition} We have $b_o(p) \geq b(p)+1$, for $p \leq 397$.
\label{When is bop greater than bp}
\end{proposition}
\begin{proof} We will show the calculation for $p=7$. The calculation is nearly identical for the other primes in question. We note that $b(7)=15.$ Now, assume $N$ is an odd perfect number with smallest prime divisor 7, and with exactly 15 distinct prime divisors. Then we have $$2=\frac{\sigma(N)}{N} < H(N) \leq \frac{7}{6}\frac{11}{10}\frac{13}{12}\frac{17}{16} \cdots \frac{53}{52}\frac{59}{58}\frac{10^8+1}{10^8} <1.994.$$
This is a contradiction.
\end{proof}
This proposition stops at 397 because the relevant product is actually greater than 2  for the next prime, 401. The result could be extended if the result from \cite{Goto and Ohno} could be extended further; however, extending that result (say to that an odd perfect number must have a prime divisor which is at least $10^9$) would likely take either very heavy new computations or would take some fundamental new insight. That said, it is  plausible that a similar result could be proved just for odd perfect numbers not divisible by any prime less than some bound, and this would allow one to extend the above proposition in this case.
We also have as a consequence of Theorem \ref{b(p) bound theorem}:
\begin{corollary}\label{Almost beat Grun} If $p \geq 11$, then $b(p) \geq 2p+2$.
\end{corollary}
Using Proposition \ref{When is bop greater than bp},  Corollary \ref{Almost beat Grun} and the earlier remarks for $p=3$ and $p=5$, we can combine this with Theorem \ref{b(p) bound theorem} bound to obtain with a little work:

\begin{corollary}\label{Grun beating linear inequalities}  Let $N$ be an odd perfect number with smallest prime factor $p$. Then we have $\omega(N)\geq 2p+2$.
\end{corollary}
\begin{proof}  The result is essentially just Corollary \ref{Almost beat Grun} except for $p=3,5,7$. $p=3$ is handled since an odd perfect number must have at least 10 distinct prime factors and $10 \geq 2(3)+2$. Since an odd perfect number not divisible by 3 must have at least 12 prime factors, 5 is likewise handled. Since $b(7)=15$, we have that $b_o(7) \geq 16$ by Proposition \ref{When is bop greater than bp}. And so the result is proven.
\end{proof}
Note that Corollary \ref{Grun beating linear inequalities}  is tighter than Grun's result for all odd primes $p$. One can see from examples like $945$ that this bound really does require that $N$ is an odd perfect number, unlike Grun's bound which applies also to odd abundant numbers.
It is easy to see from the definition of $a(n)$ that $a(n+1) \geq a(n)$ for all $n \geq 2$. However, Norton's bounds do not appear by themselves to be tight enough to  conclude that $a(n+1)> a(n)$ for all $n \geq 2$. But we can use Corollary \ref{Almost beat Grun} to prove this result.
\begin{proposition} For all $n \geq 2$ We have $a(n+1) \geq a(n)+1$. Equivalently, if $P_n$ is an odd prime and $P_{n+1}$ is the next prime after $P_n$ then $b(P_{n+1}) \geq  b(P_n)+1$.
\end{proposition}
\begin{proof} We can verify that the statement is true for any prime $p \leq 17$, so we may without loss of generality assume that $P_{n+1} > P_n \geq 19$. Assume that $b(P_n)=b(P_{n+1})=m$, or equivalently that $P_n=P_{n+1}=m$. This means we have that

\begin{equation} \prod_{r=n}^{n+m-2} \frac{P_r}{P_r-1} < 2<  \prod_{r=n}^{n+m-1} \frac{P_r}{P_r-1}. \label{Norton-Grun definition again}
\end{equation}
and
\begin{equation} \prod_{r=n+1}^{n+m-1} \frac{P_r}{P_r-1} < 2<  \prod_{r=n+1}^{n+m} \frac{P_r}{P_r-1}. \label{Norton-Grun if two values identical }
\end{equation}
We have then \begin{equation}\label{pn+1 product in terms of pn product} \prod_{r=n+1}^{n+m} \frac{P_r}{P_r-1} = \left(\prod_{r=n}^{n+m-2}\frac{P_r}{P_r-1}\right)\left(\frac{P_n-1}{P_n}\right)\left(\frac{P_{n+m-1}}{P_{n+m-1}-1}\right)\left(\frac{P_{n+m}}{P_{n+m-1}-1}\right). \end{equation}
However, we have from Equation that \ref{Norton-Grun definition again} the first term on the right hand-side of Equation \ref{pn+1 product in terms of pn product} is less than 2. We claim that the remaining terms are less than 1, which would mean that the right-hand side of Equation \ref{Norton-Grun if two values identical } would be both greater than 2 and less than 2 which is a contradiction. It just remains to show that $$\left(\frac{P_n-1}{P_n}\right)\left(\frac{P_{n+m-1}}{P_{n+m-1}-1}\right)\left(\frac{P_{n+m}}{P_{n+m-1}-1}\right) <1.$$
We note that $b(P_n) \geq 2P_n+2$,  and thus $P_{n+m_1} \geq 2p+1$. We have then $P_{n+m} \geq 2P_n+3$. Thus we have that $$ \left(\frac{P_n-1}{P_n}\right)\left(\frac{P_{n+m-1}}{P_{n+m-1}-1}\right)\left(\frac{P_{n+m}}{P_{n+m-1}-1}\right) \leq \left(\frac{x-1}{x}\right)\left(\frac{2x+1}{2x}\right)\left(\frac{2x+3}{2x+2}\right),$$ where $P_n=x.$ However, we have that $$\left(\frac{x-1}{x}\right)\left(\frac{2x+1}{2x}\right)\left(\frac{2x+3}{2x+2}\right) = \frac{4x^3 + 4x^2 -5x -3}{4x^3 +4x^2}  <1.$$ This completes the proof.
\end{proof}
In a similar vein, one can easily modify the above proof to conclude:
\begin{proposition} For any constant $c$ there are only finitely many $n$ where $a(n+1)-a(n) \leq c$.
\end{proposition}

We have just shown that $a(n)$ is increasing in $n$. This is the same as saying that the first difference, $a(n+1)-a(n)$, is always positive. One might naturally wonder about the behavior of the second differences. Since $a(n)$ asymptotically behaves like $\frac{1}{2}n^2\log n$ which has positive and indeed slightly increasing second differences, one might hope that $a(n)$ at least has always positive second differences. Alas, this is not the case. Let $f(n)$ be the second difference of $a(n)$, that is, $f(n)=a(n+1)+a(n-1)-2a(n)$. Generally, $f(n)$ is positive. However, $f(31)=-5$, and $f(100)=-144$. What is going on here? The key issue appears to be that both of these values correspond to primes which occur right after a large gap. We say that a prime $P_n$ occurs after a record setting gap if $P_n-P_{n-1}$ is larger than it is for any other choice of smaller $n$.
In particular, the 30th prime is 113, and then there is a record-setting gap to the 31st prime of 127. Similarly, the 99th prime is 523 and then there is a record setting gap until the 100th prime of 541. This shouldn't be surprising. Because there are unusual gaps here, $a(30)$ and $a(99)$ need to be extra large since the relevant products lack any smallish primes other than $P_{30}$ and $P_{99}$. (Remember the smaller a prime the more it contributes to our product.) We can check this intuition by looking at when $f(n)=1$ and noting the two smallest examples of this occur at $n=10$, corresponding to the record setting gap between 23 and 29, and at $n=25$, corresponding to the record setting gap between 83 and 89. Note that we can have $f(n)=1$ when the gap is not a record setting gap, such as at $n=35$, which corresponds to the large but not record setting gap between 139 and 149.  This discussion leads to four questions about the behavior of $f(n)$ and one about $a(n+1)-a(n)$.
\begin{enumerate}
    \item Are there infinitely many values of $n$ where $f(n)$ is negative?
    \item Is the set of $n$ where $f(n)<0$ a subset of those $n$ where $P_n$ is right after a record setting gap?
    \item Are there infinitely many $n$ where $P_n$ occurs after record setting gaps and $f(n)$ is positive?
    \item Does $f(n)$ take on any integer value? In particular, is $f(n)$ ever zero?
    \item Does  $a(n+1)-a(n)$ take on every positive integer value?
 
\end{enumerate}
\section{Hybrid bounds}
We wish to combine the Norton type results together with the Ochem-Rao type results to get a strong lower bound on the size of an odd perfect number in terms of its smallest prime factor.
We will write  $b_2(p)=b(p) + \pi(p)-1$.

Let $S$ be a set of odd primes. We say $N$ is an {\emph{S-avoiding OPN}} if $N$ is an odd perfect number not divisible by any prime in $S$. Notice in particular that if the smallest prime factor of $N$ is $p$, then $N$ is an $S$ avoiding OPN with $S$ the set of odd primes strictly less than $p$.
Given $S$ a set of primes (possibly empty), and $\alpha$, and $\beta$ to be real numbers, we will write OR($\alpha$, $\beta$, $S$)
for the statement ``For any $S$-OPN, we have  $\Omega(N) \geq \alpha \omega + \beta$.''
In this framework, Ochem and Rao's original result of Equation\ref{OR1} is  the statement OR($\frac{18}{7}$, $\frac{-31}{7}$, $\emptyset$). Similarly, Equations \ref{firstineq} and \ref{secondineq} can be stated as OR($\frac{8}{3}$,$\frac{-7}{3}$, $\{3\}$)  and OR($\frac{21}{8}$, $\frac{-39}{8}$, $\emptyset $). Theorem \ref{OR stage II} can be stated OR($\frac{302}{113}$,$\frac{-286}{113}$,$\{3\}$) and OR($\frac{66}{25}$,$-5$,$\emptyset$).

\begin{theorem}\label{Hybrid general form} Let $S$ be a finite set of odd primes. Let $\alpha$ and $\beta$ be real numbers with $\alpha >2$. Let $M$ be the maximum of $S$. Assume that $p > M$. Let $N$ be an  odd perfect number with smallest prime factor $p$, and also satisfying $\alpha\omega(N) + \beta \geq 0)$. Set $Q=P_{n+b(p)-1 }$. Then we have:
$$\log N \geq (\log p)\left((\alpha -2)(b(p)) - \beta +1\right) +  2(\vartheta(Q)- \vartheta(p)) - \log Q. $$
\end{theorem}
\begin{proof} Assume as given and note that every prime factor of an odd perfect number except possibly the special prime must be raised to at least the second power. This contributes the $2(\vartheta(P)- \vartheta(p)) - \log Q$ term (where in the worst case scenario $Q$ is the special prime). However, we have an additional contribution of the remaining primes which are forced by our lower bound for $\Omega(N)$. Each of those primes is at least $p$, and there are at least $((\alpha -2)\omega(N)+ \beta + 1$  such primes (with the +1 coming from our special prime only being raised to the first power rather than the second). This give us the other term above.
\end{proof}

We will need two  following results from \cite{Rosser}.

\begin{equation} x \left(1 - \frac{1}{2 \log x}\right) < \vartheta(x) < x  \left(1 + \frac{1}{2 \log x}\right). \label{Rosser Chebyshev bound}
\end{equation} Here $\vartheta(x)$ is Chebyshev's second function, that is $\vartheta(x) = \sum_{p \leq x} \log p$, and the upper bound is valid for $x>563$ and the lower bound is valid for $x>1$.
We have as an immediate corollary of Equation \ref{b(p) bound with Pi term}
\begin{corollary} Let $N$ be an odd perfect number with smallest prime factor $p$. Then we have \begin{equation}  b_2(p) \geq \frac{p^2}{2\log p}\left(1 - \frac{0.754}{\log p}  - \frac{0.745}{(\log p)^2} - \frac{0.247}{(\log p)^3} + \frac{0.631813}{(\log p)^4}\right)
\label{b_2(p) bound}
\end{equation}
\end{corollary}



We can use this sort of result to get results stronger than Norton's lower bounds for $\log N$ in Equation \ref{Norton 4-3} Equation \ref{Norton's improved version of Theorem 3}.
We have using our previous bounds and a little algebra the following:
\begin{lemma} Let $p$ be an odd prime greater than 3 . Set $t=\log p$ Then \begin{equation} P_{b_2(p)} \geq p^2I_3(t) \end{equation} where
\begin{equation}I_3(t) =   1 - \frac{0.754}{t} - \frac{2.5 \log t}{t^2} - \frac{1.808}{t^2} - \frac{0.55 \log t}{t^3 } +  \frac{0.41 (\log t)^2}{t^4} + \frac{0.2 \log t}{t^4} + \frac{3.6}{t^4} .
\end{equation}
\label{P_b2 lemma}
\end{lemma}
We can use Lemma \ref{Hybrid general form} and Lemma \ref{P_b2 lemma} with our bound for $b(p)$ from Theorem \ref{b(p) bound theorem} as well as the statement OR($\frac{8}{3}$,$\frac{7}{3}$, $\{3\}$) and Ochem and Rao's bound that $N>10^{1500}$ to obtain:
\begin{theorem} Let $N$ be an odd perfect number with smallest prime divisor $p$. Then we have that
\begin{equation} \log N \geq p^2 \left(\frac{7}{3}- \frac{2.51}t - \frac{2.5 \log t}{t^2} - \frac{1.31}{t^2}  -\frac{3.2 \log t}{t^3} -\frac{4.1 \log t}{t^4}  \right).
\end{equation}
\end{theorem}
Note that we have used OR($\frac{8}{3}$,$\frac{7}{3}$, $\{3\}$) rather than our new bound since our main theorem is not better until $\omega \geq 34$.  One can derive a similar result, using the main theorem of this paper which will be weaker when $N$ is divisible by a small prime $p$.

\section{On the strength of restrictions about an odd perfect number}
At this point, there are many different bounds on odd perfect numbers. These include bounds on the size of the odd perfect number in terms of its number of prime factors, bounds on the size of the largest prime factor, bounds on the size of the smallest component and bounds on the size of $N$ itself. For a given set of positive integers $A$, we will write $A(x)$ to be the number of elements in $A$ which are at most $x$. Set $E$ to be the set of numbers of Euler's form for an odd perfect number. That is,
$n \in E$ if $n =p^am^2$ where $p$ is prime, $p \equiv a \equiv 1$ (mod 4), and $(p,m)=1$
Let $P$ be a given property of a positive integer. We will write $E_P$ to be set of elements of $E$ satisfying $P$. We will say that $P$ is a strong property if the density of $E_P$ in $E$ is 0, that is
$$\lim_{x \rightarrow \infty} \frac{E_P(x)}{E(x)}=0.$$
We will similarly say that $P$ is a weak property if
$$\lim_{x \rightarrow \infty} \frac{E_P(x)}{E(x)}=1.$$\
Note for example that for any constant $k$, all of the following are weak properties:
\begin{itemize}
\item ``A number must be at least $k$''
\item ``A number must have a prime factor at least $k$''
\item ``A number must have a component at least $k$''
\item ``A number must have at least $k$ distinct primes factors.''
\item ``A number must have at least $k$ total prime factors.''
\end{itemize}
Note any finite set of weak properties cannot prove that no odd perfect numbers exist.

However, Ochem and Rao's inequality is in fact a strong property. Define $OR_{\alpha, \beta}(n)$ to be the sentence ``$\Omega(n) \geq \alpha \omega(n) + \beta$.''  It is a not difficult consequence of Theorem 430 in \cite{Hardy and Wright} to show the following:
\begin{theorem}\label{Ochem and Rao is strong} Let $\alpha$ and $\beta$ be real numbers. Assume that $\alpha > 2$. Then $OR_{\alpha, \beta}$ is a strong property.
\end{theorem}
We will say that a property $P$ is {\emph{substantially stronger}} than property $Q$ if two conditions hold:
\begin{enumerate}
    \item Every element of $E$ which is satisfied by $Q$ is satisfied by $P$.
    \item The set $E_P$ has density zero in the set $E_Q$. That is, $$\lim_{x \rightarrow \infty} \frac{E_P(x)}{E_Q(x)}=1.$$
\end{enumerate}
We then conjecture that:
\begin{conjecture} Let $\alpha_1, \alpha_2$, $\beta_1, \beta_2$ be real numbers with $\alpha_1 > \alpha_2 > 2$. Then $OR_{\alpha_1,\beta_1}$ is substantially stronger than $OR_{\alpha_2,\beta_2}$.
\end{conjecture}
Of course, any result of the form ``For any odd perfect $N$ $N$ must satisfy $OR_{\alpha, \beta}$'' cannot by itself resolve the fundamental open question, but we suspect that the strength of Ochem and Rao's result in the sense above is a sign that this is a potentially fruitful direction for further research.
We note that something being a strong property does not always line up with our intuition about what should be a ``strong'' property in a general sense. For example, let $f(x)$ be a function which is increasing for sufficiently large $x$ and satisfying that $\lim_{x \rightarrow \infty} f(x) = \infty$. Then it is not hard to show that the property $P_f$ ``For all $n$, $n$ has a prime factor which is smaller than $f(n)$'' is always a weak property. But if one could show that an odd perfect number had to have prime factor always less than $\log \log \log \log n$, that would certainly be noteworthy!

\section{Future work and related problems}
One major direction for improving these results is to prove there are no triple threats. Proving there are no triple threats would result in substantial tightening of both the bounds for the case of $3|N$ and for the case of $3 \not |N$.  Another natural object of study in this context would be what we call an $n$-obstruction.

Define an $n$-{\emph{obstruction}} to be a set of primes all greater than $3$ $a_i,b_i,c_i$ for $1 \leq i \leq n$ and $p$ an odd prime, satisfying for all $1 \leq i \leq n$
\begin{enumerate}
    \item $\sigma(a_i^2)= p\sigma(b_i^2)\sigma(c_i^2)$
    \item $\sigma(b_i^2)$ and $\sigma(c_i^2)$ prime.
    \item The $a_i$ are all distinct.
\end{enumerate}  
If we can show that a $4$-obstruction does not exist, possibly with some very small modulo restrictions we will get a substantially tighter bound. Similarly, if we can rule out $3$-obstructions or even a $2$-obstruction (although we suspect that the last isn't really doable). Note that at present we can't even show the following statement which looks like it should be obviously true:
\begin{conjecture} There exists some $n$ such that no odd perfect number $N$, contains an $n$-obstruction $a_i, b_i, c_i$ with $a_i^2 ||N$, $b_i^2 ||N$ and $c_i^2 ||N$.
\end{conjecture}

Of course, as we improve the linear term in the bounds, the general price paid is that we are subtracting more in the constant term. Thus, in the original Ochem and Rao paper, they had a constant of $-31/7$, and in the subsequent paper we had as worst case constant $-39/8$. One of the original goals of Ochem and Rao's original inequality \ref{OR1}  was to assist in the proving of inequality \ref{OR2}, and there is interest in proving inequalities of the form \begin{equation}\Omega(N) \geq 2\omega(n) +C\label{OR type2 general form} \end{equation} where $C$ is reasonably large. At present, the best such inequality is that by Ochem and Rao where $C=51$. \par
Inequalities of that form require extensive computation, where one needs to check many cases with branching in essentially the standard approach to heavy computations to bound odd perfect numbers; however Ochem and Rao had as one of their conditions to terminate a branch that Equation \ref{OR1} forced Equation \ref{OR2}. Obviously, that sort of termination will be more common when one has not just a stronger linear term but a stronger constant term. Using the inequalities from this paper to prove inequalities for the form of inequality \ref{OR type2 general form} would be easier with less negative constants. For specific small values of $\omega$ our inequalities will already give slightly better bounds than used here, but other approaches might improve the constants. \par
One might hope to use the results of  Nielsen bounding the actual size of an odd perfect number. We present here an approach that is too weak to be useful by itself but might be productive with more work. We will restrict this discussion under the assumption where we have both $5|N$ and $11|N$ where this approach is most likely to work. Assume further that we have $\omega=10$ which is the smallest possible value of $\omega$ not yet ruled out. Note that if $5|\sigma(11^{f_{11}})$ or $11|\sigma(5^{f_5})$, then we can already improve our constant term that way, so we will assume that neither of those occurs. In that case we have \begin{equation}\label{Attempt to use Nielsen1} (5^{f_5})^2(11^{f_{11}})^2 <  (5^{f_5})\sigma(5^{f_5})11^{f_{11}}\sigma(11^{f_{11}}) \leq N. \end{equation} Nielsen \cite{Nielsen} has proved that if $N$ is an odd perfect number with $k$ distinct prime factors and largest prime divisor $P$ then \begin{equation}\label{Nielsen upper bound} 10^{12}P^2N < 2^{(4^k)}. \end{equation} Combining Equation \ref{Attempt to use Nielsen1} with Nielsen's upper bound \ref{Nielsen upper bound}, as well as the fact that the largest prime factor of an odd perfect number must be at least $10^8$ by \cite{Goto and Ohno} we get that \begin{equation}\label{Linear restriction from Nielsen} (5^{f_5})^2(11^{f_{11}})^2 10^{28} < 2^{(4^{10})} \end{equation} which when we take logarithms simplifies to \begin{equation}(2 \log_2 5)f_5 + (2\log_2 11)f_{11} + 28\log_2 10 <4^{10}  \end{equation} which is a linear inequality restricting $f_5$ and $f_{11}$ but it is much too weak to give a useful restriction for improving the constant. \par
There appear to be four possible approaches to improving this inequality. The first approach is that one could improve the size of the largest prime factor of an odd perfect number; this is a project that should be undertaken in general since it has been about a decade since the last substantial improvement on this has occurred; more recent algorithmic improvements and computational power may make this a reasonable step. Unfortunately, it is unlikely that such improvement by itself would substantially improve Inequality \ref{Linear restriction from Nielsen} since the restriction involves the logarithm of the largest prime divisor. The second approach is to improve the size of the largest prime divisor, restricted to some specific range of $\omega$; it seems very likely that with the additional assumption that $\omega=10$ or even something like $\omega \leq 15$, that one can substantially improve on the lower bound for the largest prime factor. The third possibility is to use Nielsen's general machinery which he used to prove Equation \ref{Nielsen upper bound} to incorporate specific prime powers. The fourth possibility is to improve the second inequality in Equation \ref{Attempt to use Nielsen1} by making precise the intuition that there should be a large part of $N$ which is not included in $\sigma(5^{f_5})\sigma(11^{f_{11}}).$ This last looks to be the most promising.
However, given how weak equation \ref{Linear restriction from Nielsen} is, it will likely require multiple of these approaches for it to be at all productive. Even if one improves it enough to be useful for small values of $\omega$, it will still be likely too weak to be useful for even slightly larger values of $\omega$. Luckily, all four of these approaches would be of general interest to understanding odd perfect numbers.
A slightly different approach to Nielsen's bound may also be valid. Again restricting to the situation where $5|N$ and $11|N$, we have
\begin{equation}5^{f_5}11^{f_{11}}7^{2s}7^{4t}7^{6u}q^e<N  \end{equation}
and to then proceed as before. This inequality is also still too weak to be directly useful by itself but may be combined with bounds on the size of $q$.
A major part of our improvement in the case of $3 \not | N$ depended on a specific coincidental factorization of a specific composition of cyclotomic polynomials. Further understanding of such compositions may be relevant for further understanding of odd perfect numbers.  These questions about cyclotomic polynomials may be of interest independent of anything involving perfect numbers.
We  have a conjecture that essentially says that we cannot often get so lucky that we often have such factorizations. In particular:
\begin{conjecture} Let $p$ and $q$ be distinct odd primes and let $\Phi_p(x)$ and $\Phi_q(x)$ be the $p$th and $q$th cyclotomic polynomials. Then at least one of $\Phi_p(\Phi_q(x))$ or $\Phi_q(\Phi_p(x))$ is irreducible.
\label{You don't get that lucky that often}
\end{conjecture}
We also suspect that, in some suitable sense, such compositions being reducible should occur on a set of density zero. In particular, call an ordered pair of positive integers $(m,n)$ to be a {\emph{good}} if $\Phi_m(\Phi_n(x))$ factors over the integers where $\Phi_m$ and $\Phi_n$ are the $m$th and $n$th cyclotomic polynomials. Let $D(t)$ count the number of good pairs with both $m \leq t$ and $n \leq t$. Then we strongly suspect that:
\begin{conjecture} \label{You really don't get that lucky that often}
$$\lim_{t \rightarrow \infty}\frac{D(t)}{t^2}=0.$$
\end{conjecture}
Moreover, we have the following stricter version:  let $f(t)$ and $g(t)$ be strictly increasing functions which go to infinity as $t$ goes to infinity, and let $D_{f,g}(t)$ count the number of good pairs with $m \leq f(t)$ and $n \leq g(t)$. Note that in particular $D(t)=D_{t,t}(t)$). Then
\begin{conjecture} For any such $f(t)$ and $g(t)$ we have  $$\lim_{t \rightarrow \infty}\frac{D_{f,g}(t)}{f(t)g(t)}=0.$$
\label{By most metrics, the gods are against you}
\end{conjecture}
We are uncertain if Conjecture \ref{By most metrics, the gods are against you} is true, but suspect that if it is true, proving it would be very difficult.
We can make corresponding versions of Conjectures \ref{You really don't get that lucky that often} and \ref{By most metrics, the gods are against you} that are restricted to cyclotomic polynomials arising from primes. Define $\bar{D}(t)$ to be the same as $D(t)$ but counting only the good pairs $(m,n)$ where $m$ and $n$ both prime. Define $\bar{D}_{f,g}(t)$ similarly. Then we expect that
\begin{conjecture} \label{You really don't get that lucky that often when restricted to primes}
$$\lim_{t \rightarrow \infty}\frac{\bar{D}(t)}{\pi(t)^2}=0.$$
\end{conjecture}
\begin{conjecture} For any such $f(t)$ and $g(t)$ we have  $$\lim_{t \rightarrow \infty}\frac{\bar{D}_{f,g}(t)}{\pi(f(t))\pi(g(t))}=0.$$
\label{By most metrics, the gods are against you when restricted to primes}
\end{conjecture}
Note that similar questions have been asked and answered about general polynomials. See in particular \cite{Sha} and \cite{Reis}.
Ochem and Rao type results also show that many divisors of a positive integer must have many repeated prime factors. It is therefore of interest whether this sort of result can be used to improve on results like \cite{Nielsen old} which rely heavily on inducting on the divisors of an odd perfect number.
One other obvious question is whether anyone can replace the Ochem and Rao type results with a better than linear inequality. The methods used in this paper do not seem to have any hope of doing so, but it is plausible that sieve theoretic methods could result in some similar type of restriction.
One obvious question is how well we can upper bound $\Omega(N)$ in terms of $\omega(N)$. Recall Nielsen's result\cite{Nielsen old} that if $N$ is an odd perfect number then \begin{equation}\label{Nielsen bound} N < 2^{4^{\omega(N)}}.
\end{equation}
If $N$ is an odd perfect number then we trivially $3^{\Omega(N)} < N$, which when combined with Equation \ref{Nielsen bound} we obtain $$\Omega(N) < 4^{\omega(N)}\frac{\ln 2}{\ln 3}.$$ Improving this bound directly in a non-trivial fashion seems worth exploring.
Nielsen also showed \cite{Nielsen old} that if $N$ is an odd perfect number, and we have $P=\prod_{p|N}p$, then
\begin{equation}\label{Nielsen P bound} N < P^{2^{\omega(N)}}, \end{equation} from which it follows that we have $a_i \leq 2^{\omega(N) -2} $ for at least one of the $a_i$. It may be possible to use this fact to improve the Ochem-Rao results further.
We may also combine the Ochem and Rao type bounds to get a straightforward upper bound for $N$ in terms of $\omega$. In particular, if we know that $\Omega \geq \alpha \omega + \beta $ then we easily have from Equation \ref{Nielsen bound} that
\begin{equation} \label{Nielsen with OR bound general form} N < 2^{4^{(\frac{\Omega -\beta}{\alpha})}}.
\end{equation}
Using our main theorem we have the result  that if $(3,N)=1$ that
\begin{equation} \label{Nielsen with OR II bound}  N < 2^{4^{\left(\frac{113\Omega +286}{302}\right)}}. \end{equation}
It seems worth wondering if we can substantially improve upper bounds on $N$ in terms of $\Omega$ which are better than simply combining the Nielsen bound with the best available Ochem and Rao type bound.

Ochem and Rao also used similar techniques in their proof that an odd perfect number must have a component of size at least $10^{62}$.\cite{OchemRao2012}  In particular, they first showed that any odd perfect number $N$ must either have a component of size greater than $10^{62}$ or that $N$ cannot be divisible by any prime less than $10^8$. They then concluded that an odd perfect number with all components smaller than $10^{62}$ can only have  primes raised to the first, 2nd, 4th or 6th powers. They then  obtained a set of linear inequalities
relating how many such primes there were and obtained a contradiction.  It is likely that the type tightened bounds in \cite{Zelinsky1} and this paper can be used to improve that type bound. \par
An additional area of interest may be to generalize the Ochem and Rao type results beyond odd perfect numbers.  Recall a number $N$ is said to be multiply perfect if $Nk=\sigma(N)$ for some $k$, and we then say that $N$ is $k$-perfect. Perfect numbers are then $2$-perfect. It is a long-standing question if the only multiply perfect odd number is $1$.  We suspect that the Ochem and Rao type results can be extended to odd multiply perfect numbers where the constant term is allowed to be a function of $k$. \par
A different generalization of perfect numbers leads to Ore harmonic numbers. Ore noted that if $N$ is a perfect number then one must have $\sigma(N)|n\tau(n)$ where $\tau(n)$ is the number of positive divisors of $n$. Ore called numbers satisfying $\sigma(N)|n\tau(n)$ harmonic numbers since they are precisely the numbers where the harmonic mean of their positive divisors is an integer.  Note that there are multiply perfect numbers which are not Ore harmonic numbers and there are Ore harmonic numbers which are not multiply perfect numbers. Ore asked if all Ore harmonic numbers are odd. It would be interesting to see if one can extend the Ochem and Rao type results to Ore harmonic numbers.
One can also generalize Ore's harmonic numbers. We will call $n$ a generalized harmonic number if $n$ satisfies $\sigma(n)|n(\tau(n))g^m$ where $m$ is some integer and $g$ is the largest odd divisor of $\tau(n)$. It again appears that all solutions here are odd, although as far as we are aware, this generalization has not been investigated in the literature. It would be interesting to see if Ochem and Rao type bounds can be extended to these generalized harmonic numbers. \par

Another interesting direction is rather than generalize instead to narrow the situation. Colton \cite{Colton} has shown that no perfect number (whether even or odd) satisfies $\tau(n)|n$. However, the set of positive integers which satisfy $\tau(n)|n$ has density zero.\cite{Kennedy} In contrast,  the set of numbers $n$ where $\tau(n)|\sigma(n)$ has density 1. \cite{BEPS} It is not hard to show that the only even perfect number $n$ satisfying $\tau(n)|\sigma(n)$ is $n=6$. One might ask if we can say anything interesting about odd perfect numbers $N$ satisfying $\tau(N)|\sigma(N)$. In particular, it is likely that Ochem-Rao type results can be substantially improved if one is restricted to this set. \\

{\bf Acknowledgements:} The author is grateful to detailed feedback from Pascal Ochem which substantially improved the presentation as well as the strength of the results. Aaron Silberstein first alerted the author to the fact that a composition of cyclotomic polynomials can be reducible. Douglas McNeil pointed out that an earlier version of Lemma \ref{Relative primeness of two parts of T to S1} was incorrect, and also pointed out that a previous version's conjecture was hopelessly false.

\end{document}